\documentclass[A4paper, oneside]{amsart}
\usepackage[utf8]{inputenc}
\usepackage{inputenc}
\usepackage[backref=page]{hyperref}
\usepackage{amsmath,amssymb,amsthm,mathtools}
\usepackage{xspace}

\hypersetup{
    colorlinks=true,
    linkcolor=red,
    filecolor=magenta,      
    urlcolor=cyan,
    citecolor=blue,
    pdftitle={Overleaf Example},
    pdfpagemode=FullScreen,
    }

\usepackage[usenames,dvipsnames]{xcolor} % allows you to use color names, call this BEFORE you call TikZ
\usepackage{tikz}
\usepackage{pgfplots}
\usepackage{array}
\usetikzlibrary{calc}
\usetikzlibrary{shapes}
\usetikzlibrary{matrix}
\usetikzlibrary{intersections}
\usetikzlibrary{patterns}
\usetikzlibrary{positioning}
\usetikzlibrary{arrows.meta}
\usetikzlibrary{bending}
\usetikzlibrary{shapes.geometric}

\usepackage{graphicx}
\usepackage{caption}
\usepackage{booktabs}
\usepackage{multirow}
\usepackage{array}
\usepackage{adjustbox}

\usepackage{amsthm}
\usepackage[shortlabels]{enumitem}

\newcommand{\Red}{\operatorname{Red}}

\newcommand{\N}{\mathbb{N}}
\newcommand{\Z}{\mathbb{Z}}

\newcommand{\id}{\textit{e}}
\newcommand{\mpair}[1]{\langle\, #1\,\rangle}  % for group generated by S with relations or scalar product
\newcommand{\ord}{\operatorname{ord}}
\newcommand{\Cay}{\operatorname{Cay}}
\newcommand{\oCay}{%
  \operatorname{%
    \overset{\smash{\raisebox{-0.45ex}{$\longrightarrow$}}}{\textnormal{Cay}}%
  }%
}
\newcommand{\cyc}{\operatorname{cyc}}

\newcommand{\Sb}{\mathbb{S}}
\newcommand{\E}{\mathbb{E}}
\newcommand{\Hb}{\mathbb{H}}
\newcommand{\X}{\mathbb{X}}
\newcommand{\T}{\mathbb{T}}
\newcommand{\W}{\mathcal{W}}
\newcommand{\A}{\mathcal{A}}
\newcommand{\V}{\mathcal{V}}
\newcommand{\LL}{\mathcal{L}}
\newcommand{\Irr}{\mathcal{I}}
\newcommand{\U}{\mathcal{U}}
\newcommand{\G}{\mathcal{G}}

\newcommand{\EIS}{\textnormal{\bf EIS}\xspace}
\newcommand{\EMIS}{\textnormal{\bf EMIS}\xspace}
\newcommand{\MIS}{\textnormal{\bf MIS}\xspace}
\newcommand{\IS}{\textnormal{\bf IS}\xspace}
\newcommand{\QCS}{\textnormal{\bf QCS}\xspace}
\newcommand{\CS}{\textnormal{\bf CS}\xspace}
\newcommand{\tworec}{\textnormal{\bf 2-Rec}\xspace}
\newcommand{\Isom}{\text{Isom}}

\newcommand{\RomanNumeralCaps}[1]
{\MakeUppercase{\romannumeral #1}}

\newtheorem{thm}{Theorem}[section]
\newtheorem{cor}[thm]{Corollary}
\newtheorem{lem}[thm]{Lemma}
\newtheorem{prop}[thm]{Proposition}
\newtheorem{question}{Question}

\theoremstyle{definition}
\newtheorem{rem}[thm]{Remark}
\newtheorem{defi}[thm]{Definition}
\newtheorem{ex}[thm]{Example}

\tikzset{
	sommet/.style={circle,draw,fill=blue,inner sep=1.5pt,thick},
	diagramme/.style={baseline=-0.5ex}
}

\newcommand{\edgeAB}[1]{%
	\begin{tikzpicture}[diagramme,scale=.8]
		\node[sommet](a) at (0,0) {};
		\node[sommet](b) at (.8,0) {};
		\draw[thick](a)--node[above,scale=.7]{$#1$}(b);
\end{tikzpicture}}

\newcommand{\chainABC}[2]{%
	\begin{tikzpicture}[diagramme,scale=.8]
		\node[sommet](a) at (0,0) {};
		\node[sommet](b) at (.8,0) {};
		\node[sommet](c) at (1.6,0) {};
		\draw[thick](a)--node[above,scale=.7]{$#1$}(b);
		\draw[thick](b)--node[above,scale=.7]{$#2$}(c);
\end{tikzpicture}}

\newcommand{\triangleABC}[3]{%
	\begin{tikzpicture}[diagramme,scale=.7]
		\useasboundingbox (-1.2,-.8) rectangle (1.2,1.2);
		\node[sommet](a) at (90:.7) {};
		\node[sommet](b) at (210:.7) {};
		\node[sommet](c) at (330:.7) {};
		\draw[thick](a)--node[left,scale=.7]{$#1$}(b);
		\draw[thick](b)--node[below,scale=.7]{$#2$}(c);
		\draw[thick](c)--node[right,scale=.7]{$#3$}(a);
\end{tikzpicture}}

\newcommand{\triangleABCC}[2]{%
	\begin{tikzpicture}[diagramme,scale=.7]
		\useasboundingbox (-1.2,-.8) rectangle (1.2,1.2);
		\node[sommet](a) at (90:.7) {};
		\node[sommet](b) at (210:.7) {};
		\node[sommet](c) at (330:.7) {};
		\draw[thick](a)--node[left,scale=.7]{$#1$}(b);
		\draw[thick](b)--node[below,scale=.7]{$#2$}(c);
      \end{tikzpicture}}

\newcommand{\nodesA}{%
	\begin{tikzpicture}[diagramme,scale=.8]
		\node[sommet] at (0,0) {};
\end{tikzpicture}}

\newcommand{\nodesAB}{%
	\begin{tikzpicture}[diagramme,scale=.8]
		\node[sommet] at (0,0) {};
		\node[sommet] at (.8,0) {};
\end{tikzpicture}}

\newcommand{\nodesABC}{%
	\begin{tikzpicture}[diagramme,scale=.8]
		\node[sommet] at (0,0) {};
		\node[sommet] at (.8,0) {};
		\node[sommet] at (1.6,0) {};
\end{tikzpicture}}

\newcommand{\splitABC}[1]{%
	\begin{tikzpicture}[diagramme,scale=.8]
		\node[sommet](a) at (0,0) {};
		\node[sommet](b) at (.8,0) {};
		\node[sommet](c) at (1.6,0) {};
		\draw[thick](b)--node[above,scale=.7]{$#1$}(c);
\end{tikzpicture}}

\title{Weak order on groups generated by involutions}

\author[F. Dos Santos]{Fabricio Dos Santos}
\address[Fabricio Dos Santos]{Department of Mathematics and Statistics, McGill University, Burnside Hall, 805 Sherbrooke Street West, Montr\'eal, Qu\'ebec, H3A 0B9, Canada}
\email{fabricio.dossantos@mail.mcgill.ca}

\author[C. Hohlweg]{Christophe~Hohlweg}
\address[Christophe Hohlweg]{Universit\'e du Qu\'ebec \`a Montr\'eal\\
LaCIM et D\'epartement de Math\'ematiques\\ CP 8888 Succ. Centre-Ville\\
Montr\'eal, Qu\'ebec, H3C 3P8\\ Canada}
\email{hohlweg.christophe@uqam.ca}
\urladdr{http://hohlweg.math.uqam.ca}

\author[A. Trufanov]{Aleksandr Trufanov}
\address[Aleksandr Trufanov]{D\'epartement de math\'ematiques et de statistiques, Universit\'e de Montr\'eal, Montr\'eal,
Qu\'ebec, H3C 3J7, Canada}
\email{trufaleks2022@gmail.com}

\thanks{This work was partially supported by the NSERC  grant {\em combinatorics of infinite Coxeter groups} held by Hohlweg. Aleksandr Trufanov was partially funded by a excellence fellowship (“365478”) from the Fonds de recherche du Qu\'ebec.}

\begin{document}

\begin{abstract} In this article, we propose to initiate the general study of involution systems.  An {\em involution system}, that is, a group $W$ generated by a set of involutions $S$, is naturally endowed with a {\em weak order} arising from orienting the Cayley graph of $(W,S)$. In the case of a Coxeter system $(W,S)$, Bj\"orner showed that the weak order is a complete meet-semilattice. This fact has many important consequences for Coxeter systems and their related structures. In this article, we discuss the following question: For which involution systems is the weak order a complete meet-semilattice?

The class of involution systems that satisfies this condition is larger than the class of Coxeter systems (it contains, for instance, Cactus groups). In the case of an involution system with sign character, we provide a finite presentation by generators and relations and a classification in rank 3. We also obtain new characterizations of Coxeter systems in terms of the weak order, and prove a number of results on certain subclasses of these involution systems. Finally, we discuss further works and open problems in relation to biautomatic structures, geometric representations, mediangle graphs, and more.
\end{abstract}

\maketitle

\setcounter{tocdepth}{1}
\tableofcontents 

\section{Introduction} Let $(W,S)$ be an {\em involution system}, that is, a group $W$ generated by a set of involutions $S$; we only consider finitely generated involution systems in this article, i.e., $S$ is finite. 

In \cite[Ch. IV, \S1.1-\S1.3]{Bo68}, groups generated by involutions are mentioned, and first properties of their length functions are stated. The aim of Bourbaki's treatment was to provide necessary and sufficient conditions for an involution system to be a Coxeter system. The focus on Coxeter systems seems to have overshadowed the study of involution systems for themselves as abstract groups, although numerous examples appear in the literature: for instance as {\em string $C$-groups} of regular abstract polytopes, see~\cite{McSc02} or \cite{Lee21}; {\em cactus groups}, see~\cite{BlowUps,HeKa06, Devadoss};  {\em  groups generated by point reflections} and Riemannian symmetric spaces, see~\cite{Lo69}; or recently {\em super Weyl groups} associated to Lie superalgebras~\cite{WeylSuper24}.

 In this article, we propose to initiate the general study of involution systems. Involution systems are naturally endowed with the {\em weak order} $(W,\leq_R)$, a poset defined by $u\leq_R v$ if a reduced word for $u$ is a prefix of a reduced word for $v$, or equivalently, if a geodesic in the (right) Cayley graph of $(W,S)$ from the identity $\id$ to $v$ passes through $u$. In the case of a Coxeter system, Bj\"orner~\cite{Bj84} showed that the weak order is a complete meet-semilattice (for a proof, see \cite[Theorem~3.2.1]{BjBr05}). This fact has many important consequences for Coxeter systems and their connected structures, such as in the study of:
 the {\em Tamari lattice}, {\em Cambrian fans} and {\em generalized associahedra}, see for instance~\cite{Asso12};
  {\em $0$-Hecke algebras} in the case of the symmetric group, see for instance \cite{TvW15} and the references therein; 
  the representation theory of quivers and their torsion classes~\cite{TorCl23}; 
  Garside families and Shi arrangements; see for instance \cite{DDH14,DFHM24} and the references therein; biautomaticity for Coxeter systems~\cite{OP22,Sa25}; root systems in Lie theory and reflection ordering in Kazhdan-Lusztig theory, see for instance~\cite{CePa04,Dy19}.
  
  \smallskip
   In this work, we  consider the following question: {\em For which involution systems is the weak order a complete meet-semilattice?} 
 
 \smallskip
 We start by discussing involution systems in general and some of their basic properties in  \S\ref{se:IS}. In particular, we discuss \emph{sign characters} for an involution system and define the notion of \emph{even involution systems} (\EIS for short), which are involution systems admitting a sign character. 

In \S\ref{se:WeakOrder}, we describe the weak order on involution systems. Moreover,  we give many examples of \emph{meet involution systems} (\MIS for short), i.e., involution systems whose weak order is a complete meet-semilattice. In particular, we prove that {\em cactus groups}, and more generally, involution systems with median Cayley graphs~\cite{genevois2022cactus} (called \emph{right-angled mock reflection groups} in \cite{RAMRG}) are \MIS. 

We then restrict our attention in \S\ref{se:Presentation} to \emph{even meet involution systems} (\EMIS for short), i.e., involution systems that are both \EIS and \MIS, and prove our first main result. 

\begin{thm}[Theorem~\ref{thm:presentationEMIS}] 
\label{thm:1}
Any involution system $(W,S)$ with a sign character whose weak order is a complete meet-semilattice (i.e., an \EMIS) admits the following finite presentation:
$$
W=\mpair{S\mid r^2=\id, \ ss_1\cdots s_k=tt_1\cdots t_k \, \forall r \in S \text{ and } s,t\in S\text{ bounded in } (W,\leq_R)},
$$  
where $ss_1\cdots s_k=tt_1\cdots t_k$ are reduced words for the join $s\vee_R t$ of bounded $s,t\in S$.
\end{thm}

The proof is based on the study of {\em irreducible cycles} in the Cayley graph $\G$ of an \EMIS. These are elements of the cycle space of $\G$ (or, equivalently, the first singular homology of $\G$ over $\Z_2$) that cannot be decomposed into a sum of cycles of strictly smaller length. In particular, we show in Proposition \ref{weakly-lemma} that the set of irreducible cycles of $\G$ \emph{weakly intersects}, meaning that the intersection of any two distinct irreducible cycles does not contain consecutive edges (in fact, we prove something stronger in Corollary \ref{cor: stronger weak intersection}). Then, to prove Theorem~\ref{thm:1}, we show that the relations of an \EMIS are given by the words labeling all irreducible cycles containing $\id$ (Proposition~\ref{prop:RelationsIrred}), and weak intersection guarantees finiteness of the presentation and highly restricts the possible words labeling such cycles. The case of finding a presentation of a meet involution system without a sign character is open; see \S\ref{ss:Alt5} for such an example.

Observe also that the presentation in Theorem~\ref{thm:1} provides a natural  Artin-Tits group analog for an \EMIS $(W,S)$ that is obtained by removing the relations $s^2$ for each $s \in S$. Such groups have been studied by Scott in \cite{RAMRG} under the name of \emph{mock Artin groups} when the Cayley graph of $(W,S)$ is median (i.e., when the \EMIS in question is a right-angled mock reflection group). The Artin-Tits group of a general \EMIS remains to be studied.

As an easy corollary to Theorem~\ref{thm:1}, we obtain a new characterization for an involution system to be a Coxeter system in terms of the weak order.

\begin{thm}[Corollary~\ref{cor:CoxeterCharac}]
	Let $(W,S)$ be an involution system. Then the following conditions are equivalent.
	\begin{enumerate}
		\item $(W,S)$ is a Coxeter system;
		\item $(W,S)$ is an \EMIS and for any $s,t\in S$ bounded, the reduced words $ss_1 \cdots s_k$ and $tt_1 \cdots t_k$ representing $s \vee_R t$ (as in Theorem~\ref{thm:1}) have their letters in $\{s,t\}$.
	\end{enumerate}
\end{thm}

We also introduce in \S\ref{se:Presentation} the notions of \emph{companion graph} and \emph{Coxeter companion} of an \EMIS $(W,S)$. In particular, the companion graph of $(W,S)$ is a graph encoding for each pair $s,t \in S$ the length of the unique irreducible cycle containing the edges $\{\id,s\}$ and $\{\id,t\}$ in the Cayley graph of $(W,S)$. The Coxeter companion of $(W,S)$ is then defined as the Coxeter system having as Coxeter type the companion graph of $(W,S)$. We also define the Coxeter type of an involution system in a similar way as for a Coxeter system (see Definition~\ref{def:CoxeterGraphIS}). This leads to another new characterization of Coxeter systems.

\begin{thm}[Corollary \ref{cor:CompanionCox}]
	Let $(W,S)$ be an \EMIS. Then $(W,S)$ is a Coxeter system if and only if its Coxeter type and companion graph are the same. 
\end{thm}

In \S\ref{se:2recognizable}, we introduce the notion of a {\em $2$-recognizable presentation} (Definition~\ref{def: 2-rec subsets}) as a means to describe the words appearing in Theorem~\ref{thm:1}. The idea is to encode the weak intersection of irreducible cycles in terms of combinatorics of words. Roughly speaking, a $2$-recognizable is a presentation $\langle S \mid \{s^2 \mid s \in S\} \cup R_0 \rangle$ where the words of $R_0$ cannot overlap by two consecutive letters. This condition can be rephrased in terms of small cancellation theory (over free products) \cite{LyndonScupp}, see Remark~\ref{rem: small cancellation}. 

Unfortunately, it is not the case that every \emph{$2$-recognizable involution system} (i.e., an involution system given by a $2$-recognizable presentation) is an \EMIS, see \S\ref{sec: 2-rec not EMIS} for a discussion. However, we show this is true for rank $3$ in \S\ref{Rank 3 section}, and the proof depends on the classification of $2$-recognizable involution systems of rank 3 we provide below.

\begin{thm}[Theorem \ref{thm:rk3-reco}]
	\label{thm: 2-rec intro}
	Let $(W,S)$ be an involution system of rank 3. Then $(W,S)$ is a $2$-recognizable involution system if and only if it admits one of the following presentations (where $\infty$ means there is no relation): 
	\begin{enumerate}[(i)]
		\item $\langle a,b,c \, | \, a^2=b^2=c^2 = (abc)^{2m}=\id\rangle$, $m \in \N_{\geq 1}$;
		
		\item $\langle a,b,c \, | \, a^2=b^2=c^2 = (abacbc)^{m}=\id \rangle$, $m \in \N_{\geq 1}$;
		
		\item $\langle a,b,c \, | \, a^2=b^2=c^2 = (abac)^m=(bc)^n=\id\rangle$, $m \in \N_{\geq 1}$, $n \in \N_{\geq 2} \cup \{\infty\}$; 
		
		\item $\langle a,b,c \, | \, a^2=b^2=c^2 = (ab)^m=(bc)^n=(ac)^l=\id\rangle$, $m,n,l \in \N_{\geq 2} \cup \{\infty\}$ (the case of Coxeter systems of rank $3$).
		
	\end{enumerate}    
\end{thm}  

In \S\ref{Rank 3 section}, we discuss the subclass of {\em quasi-Coxeter systems}. These are involution systems $(W,S)$ having Hasse diagram of the weak order isomorphic to the Hasse diagram of the weak order of a Coxeter system. In other words, quasi-Coxeter systems are \EMIS whose weak order is lattice isomorphic to the weak order of a Coxeter system (Proposition~\ref{prop:QuasiCoxeter}). We prove, in particular, the following results about quasi-Coxeter systems.

\begin{thm} $\ $
\label{thm:2} 
\begin{enumerate}
\item The growth series of quasi-Coxeter systems are regular (Corollary~\ref{cor:regular}).
\item  Quasi-Coxeter systems are biautomatic (Theorem~\ref{thm:Automatic}).
\end{enumerate}
\end{thm}

We then show that, in rank 3, the notions discussed so far coincide, and that we can interpret these groups in a nice geometric way.

\begin{thm}[Theorem \ref{rank3-theorem}] \label{thm: rank 3 intro}
	Let $(W,S)$ be an involution system of rank $3$. Then the following statements are equivalent:
	\begin{enumerate}
		\item $(W,S)$ is an \EMIS; 
		\item $(W,S)$ admits a $2$-recognizable presentation; 
		\item $(W,S)$ admits one of the presentations in Theorem~\ref{thm: 2-rec intro}; 
		\item $(W,S)$ is a quasi-Coxeter system.
	\end{enumerate}
	Moreover, every \EMIS of rank $3$ is a discrete subgroup of isometries of the sphere, Euclidean plane or hyperbolic plane.
\end{thm}

 The proof of Theorem~\ref{thm: rank 3 intro} is based on an analysis of the $2$-recognizable presentations in Theorem~\ref{thm: 2-rec intro}. For each of these presentations, we provide a representation of $(W,S)$ as a discrete subgroup of isometries on the space $\X = \Sb^2, \E^2$ or $\Hb^2$ generated by reflections and rotations of angle $\pi$.  As a consequence, we classify all \EMIS of rank at most $3$, which are listed in Tables~\ref{table: Classification in rank 3}~and~\ref{table2: Classification in rank 3}. 

  Finally, in \S\ref{se:FurtherOpen}, we discuss further works and open problems.  In \S\ref{ss:Rank4}, we provide an example of an \EMIS of rank $4$ that is not a quasi-Coxeter system. We then discuss results and questions about growth series, (bi)automatic structures, the existence of geometric representations, as well as lattice-theoretic properties of \EMIS and parabolic subgroups. Lastly, Tatiana Smirnova-Nagnibeda and Megan Howarth have pointed out possible connections between mediangle graphs \cite{genevois2022rotation} and Cayley graphs of \EMIS, which we discuss in \S\ref{ss:Mediangleness}.

%%%%%%%%%%%%%%
\subsection*{Acknowledgments} The first  author (FDS) warmly thanks  Piotr Przytycki for many enlightening discussions and suggestions. The second author (CH) warmly thanks Angela Carnevale for discussions on weak orders on involution systems arising from a Coxeter system, Matthew Dyer for having shared his lecture notes~\cite{DyLectures} and Jon McCammond for having shared his thoughts on discrete reflection groups that are not Coxeter systems. The third author (AT) warmly thanks Leonid Rybnikov, and Tatiana Smirnova-Nagnibeda and Megan Howarth for pointing out the connections to mediangle graphs.

%%%%%%
\section{Involution systems}
\label{se:IS}

An {\em involution system} is a pair $(W,S)$ where $W$ is a group with identity $\id$ and $S\subseteq W$ is a set of involutions\footnote{By convention, an involution is an element $u\in W$ of order $2$, that is, $u\not =e$ and $u^2=\id$.} that generates~$W$. The cardinality $|S|$ is the {\em rank of $(W,S)$}. We say that the involution system $(W,S)$ is  {\em finite} (or {\em spherical}) if $W$ is a finite group. 

Amongst examples of involution systems are {\em Coxeter systems}, which we discuss in~\S\ref{ss:Coxeter} below; see for instance \cite{Hu90,BjBr05,AbBr08} for more details on Coxeter systems. This section is inspired by~\cite{DyLectures}.

\subsection{Length function} Let $(W,S)$ be an involution system. The \emph{length on $(W,S)$} is the function  $\ell_{(W,S)}: W \to \mathbb{N}$ defined by $\ell_{(W,S)}(\id)=0$ and for $w\not = \id$,
$$
 \ell_{(W,S)}(w)= \min\{k \in \mathbb{N}^* \mid w=s_1s_2\cdots s_k, \ s_i \in S\}.
$$
If there is no possible confusion, we simply write $\ell:=\ell_{(W,S)}$.

Denote by $S^*$ the free monoid on the alphabet $S$. For simplicity, we denote by  $s_1\dots s_k$ both the concatenation of letters in $S$ and the product of generators of $S$ in $W$.  A word $s_1\cdots s_k$ in $S^*$ is a {\em reduced word for $w\in W$} if   $w=s_1\dots s_k$ in $W$ and $k=\ell(w)$. 

For simplicity, we say that {\em $w=s_1\dots s_k$ is a reduced word}. We denote by $\Red(W,S)\subseteq S^*$ the set of all reduced words for $(W,S)$ and by $\Red_w(W,S)$ the set of reduced words for $w\in W$. Moreover, for $u,v,w\in W$, we say that: 
\begin{itemize}[leftmargin=\parindent]
\item {\em $w=uv$ is a reduced product}\index{reduced word!reduced product} if $\ell(w)=\ell(u)+\ell(v)$, i.e., the concatenation of any reduced word for~$u$ with any reduced word for~$v$ is a reduced word for $w$; 
\item {\em $u$ is a prefix\index{reduced word!prefix} of $w$} if a reduced word for $u$ is a prefix of a reduced word for~$w$;
\item {\em $v$ is a suffix\index{reduced word!suffix} of $w$} if a reduced word for $v$ is a suffix of a reduced word for~$w$.
\end{itemize}
Observe that if $w=uv$ is reduced, then  $u$ is a prefix of $w$ and $v$ is a suffix of $w$.  The following straightforward proposition states basic properties for the length function.

\begin{prop}\label{prop:Length} Let $(W,S)$ be an involution system and $w,w' \in W$. 
\begin{enumerate}
\item $\ell(w)=1$ if and only if $w \in S$.
\item The inverse of a reduced word is a reduced word: $\ell(w)=\ell(w^{-1})$.
\item $\ell(ww') \leq \ell(w)+\ell(w')$.
\item $\ell(ww') \geq |\ell(w)-\ell(w')|$.
\item $\ell(w)-1 \leq \ell(sw) \leq \ell(w)+1$, for any $s\in S$; or equivalently $\ell(w)-1 \leq \ell(ws) \leq \ell(w)+1$.            
\item All prefixes and suffixes of a reduced word are reduced.
\end{enumerate}
\end{prop}

The {\em left descent set $D_L(w)$} and {\em right descent set $D_R(w)$} of an element of $w\in W$ are defined by
\begin{equation}
\label{eq:Descents}
D_L(w)=\{s \in S\mid \ell(sw)<\ell(w)\}\quad \text{and}\quad D_R(w)=\{s \in S\mid \ell(ws)<\ell(w)\}.
\end{equation}

%%%%%%
\subsection{Sign character and Even Involution Systems (\EIS)}\label{ss:EIS}

It might happen that, for some involution systems,  $\ell(sw)=\ell(w)$ for some $w\in W$ and $s\in S$. For a Coxeter system, it is well-known that $\ell(sw)=\ell(w)\pm1$ for all $s\in S$ and all $w\in W$. In fact, this is a consequence of the existence of a sign character for $(W,S)$; see for instance \cite{DyLectures}.

\begin{prop}\label{prop:Sign} Let $(W,S)$ be an involution system. Then the following statements are equivalent:
\begin{enumerate}
\item $(W,S)$ admits a {\em sign character}, that is, there is a group epimorphism $\varepsilon:W\to \{\pm 1\}$ such that $\varepsilon (s)=-1$ for all $s\in S$;
\item  $\ell(sw)=\ell(w)\pm 1$ for all $w\in W$ and $s\in S$;
\item $\ell(ws)=\ell(w)\pm 1$ for all $w\in W$ and $s\in S$.
\end{enumerate}
In this case $\varepsilon(w)=(-1)^{\ell(w)}$ for all $w\in W$ and the sign character is unique.
\end{prop}
\begin{proof} Assume there is a sign character for $(W,S)$. Therefore $\varepsilon(w)=(-1)^{\ell(w)}$ for all $w\in W$. In particular, for $s\in S$ and $w\in W$, we have $\epsilon(sw)=\epsilon(s)\epsilon(w)=-(-1)^{\ell(w)}=(-1)^{\ell(w)\pm1}$. Therefore $\ell(w)\not = \ell(sw)$ and we conclude by Proposition~\ref{prop:Length}~(5). 

Conversely, assume that $\ell(sw)=\ell(w)\pm 1$ for all $w\in W$ and $s\in S$. We prove that the function $\varepsilon:W\to \{\pm 1\}$ defined by $\varepsilon(w)=(-1)^{\ell(w)}$ is a group epimorphism. Since $\varepsilon(\id)=0$ and $\varepsilon (s)=-1$ for all $s\in S$, this function is surjective. Let $w\in W$ and $s\in S$, then:
$$
\varepsilon (sw)=(-1)^{\ell(sw)}=(-1)^{\ell(sw)\pm1}=-(-1)^{\ell(w)}=\varepsilon(s)\varepsilon(w).
$$
We conclude that $\varepsilon(uv)=\varepsilon(u)\varepsilon(v)$ by induction on $\ell(u)$. The equivalence between (2) and (3) follows directly from Proposition~\ref{prop:Length}.
\end{proof}

\begin{defi}\label{def:EIS} An involution system $(W,S)$  is an {\em even involution system}, or \EIS for short, if $(W,S)$ has a sign character.
\end{defi}

Any Coxeter system is an \EIS. However, there are involution systems that do not have a sign character: for instance the alternate group $\mathcal A_5$ generated by three involutions as in~\S\ref{ss:Alt5} is not an $\EIS$.

\begin{prop}\label{prop:EIS-relations} Let $(W,S)$ be an involution system and $R\subseteq S^*$ a set of relations, i.e., $W=\mpair{S\mid R}$. Then $(W,S)$ is an \EIS if and only if any word $w$ in $R$ has an even number of letters in $S$. 
\end{prop}
\begin{proof} Assume that $(W,S)$ is an \EIS with sign character $\varepsilon$. Let $w=s_1\dots s_k\in R$, then when considering $w$ as an element of $W$, we have $\varepsilon(w)=(-1)^k=\varepsilon(\id)=1$ by Proposition~\ref{prop:Sign}. Hence $k$ is even. Conversely, consider the function $f:S\to \{\pm 1\}$ defined by $f(s)=-1$. Consider the free group $\mathbb F_S$ over $S$ and recall that $S^*$ is a submonoid of $\mathbb F_S$. By  the universal property of free groups, there is a unique epimorphism $\psi:\mathbb F_S \to \{\pm 1\}$ such that $\psi(s)=-1$ for all $s\in S$. By assumption $R\subset\ker\psi$. So there is a unique epimorphism $\varepsilon:W\to \{\pm 1\}$ such that $\varepsilon(s)=-1$ for all $s\in S$, which is a sign character by definition. 
\end{proof}

%%%%
\subsection{Involution systems and Coxeter systems}\label{ss:Coxeter}

We recall here the notion of Coxeter graphs and Coxeter systems. 

A {\em Coxeter graph} (or {\em Coxeter diagram}) is  an edge-labeled graph $\Gamma$   whose edges are labeled by elements of the set $\mathbb N_{\geq 3}\sqcup\{\infty\}=\{3,\dots, \infty\}$. By convention, the label $3$ is omitted on the diagrammatic representation of a Coxeter graph, so all edges without label are in fact edges labeled by $3$. For instance, the edges $\{1,2\}$ and $\{3,4\}$ have label $3$ in the following Coxeter graph: 
\begin{center}
$\Gamma:$ \hskip 1cm
\begin{tikzpicture}
    [scale=1,
     sommet/.style={inner sep=2pt,circle,draw=black,fill=blue,thick,anchor=base},
     rotate=0,
     baseline = 0]
    \tikzstyle{every node}=[font=\small]
    % Anchor point
    \coordinate (ancre) at (0,-0.0);
    
    % Nodes
    \node  at ($(ancre)+(0,0.6)$) {};
    \node[sommet] (a1) at ($(ancre)+(1,0.5)$) {};
    \node[above=2pt,color=blue] at (a1) {$3$};

    \node[sommet] (a3) at ($(ancre)+(2,0)$) {} edge[thick] node[above,pos=0.35] {$\infty$} (a1);
    \node[below=2pt,color=blue] at (a3) {$4$};

    \node[sommet] (a2) at ($(ancre)+(1,-0.5)$) {} 
    edge[thick] node[auto,swap,right] {} (a1) 
    edge[thick] node[auto,swap,below,pos=0.45] {$\infty$} (a3);
   \node[below=2pt,color=blue] at (a2) {$2$};
    
    \node[sommet] (a4) at (ancre) {} 
    edge[thick] node[auto,swap,right] {} (a1) 
    edge[thick] node[auto,swap,above,pos=0.35] {} (a2);
    \node[below=2pt,color=blue] at (a4) {$1$};
\end{tikzpicture}.
\end{center}

Let $S$ be a set. A {\em Coxeter matrix over $S$}  is a function  
$
m:S\times S \to \mathbb N^*\sqcup\{\infty\}=\{1,2,3,\dots, \infty\}
$  
satisfying the following conditions:
\begin{enumerate}[\itshape (i)]
\item $m(s,s)=1$ if and only if $s\in S$;
\item $m(s,t)=m(t,s)$ for all $s,t\in S$ (a Coxeter matrix is symmetric).
\end{enumerate}

A Coxeter graph $\Gamma$ with set of vertices $S$ is naturally associated to a Coxeter matrix $m_\Gamma$ over $S$. For $s,t\in S$, take the coefficient $m_\Gamma(s,t)$ to be equal to the label of the edge $\{s,t\}$ in $\Gamma$, and  if $\{s,t\}$ is not an edge of $\Gamma$, set  $m_\Gamma(s,s)=1$ and  $m(s,t)=2$ for $s\not = t$. For instance, for the above Coxeter graph $\Gamma$, we obtain the following Coxeter matrix:
$$
m_\Gamma=\begin{pmatrix}
1&3&3&2\\
3&1&3&\infty\\
3&3&1&\infty\\
2&\infty&\infty&1
\end{pmatrix}.
$$

\begin{defi}\label{def:CoxeterGraphIS} The {\em Coxeter type of an involution system} $(W,S)$ is the Coxeter graph with set of vertices~$S$ given by the Coxeter matrix over $S$ defined by $m(s,t)=\ord_W(st)$, the order of $st$ in $W$. 
\end{defi}

Recall that a {\em Coxeter system} is a pair $(W,S)$ where  $W$ is a group generated by a subset $S$ satisfying the following presentation: 
 $$
 W=\mpair{S\mid (st)^{m(s,t)}=e \text{ for }s,t\in S\text{ with }m(s,t)<\infty },
 $$
 where $m:S\times S \to \mathbb N^*\sqcup\{\infty\}$ is a Coxeter matrix over $S$.  
 It turns out that $\ord_W(s)=2$, so $S$ is a set of involutions, and $m(s,t)=\ord_W(st)$. Therefore $(W,S)$ is an \EIS and is uniquely determined (up to isomorphism) by its Coxeter graph $\Gamma$, which is called {\em the type of the Coxeter system $(W,S)$}.

\begin{rem} Contrary to the case of Coxeter systems, the Coxeter type does not determine (up to isomorphism of edge-labeled graphs) a unique involution system; see \S\ref{ss:Alt5} for an example. 
\end{rem}

\smallskip
Let $S$ be a set of cardinality $n\in\mathbb N^*$. The {\em universal Coxeter system of rank $n$} is the Coxeter system $(\mathcal U_n,S)$ of type the complete graph with vertices $S$ and all edges labeled by $\infty$. In other words, 
$
\mathcal U_n=\mpair{S\mid s^2=e}.
$
For instance, $\mathcal U_1=\mathbb Z_2$, $\mathcal U_2$ is the infinite dihedral group and for $n=3,4$ we have the following Coxeter graphs:
\smallskip

\begin{center}
      \begin{tikzpicture}[sommet/.style={inner sep=2pt,circle,draw=blue!75!black,fill=blue!40,thick}]
	\node[sommet,label=below:] (beta) at (-6,-1.25) {};
	\node[sommet,label=below:] (tau)  at (-5,-1.25){}  edge[thick] node[auto,scale=0.7,swap] {$\infty$} (beta); 
	\node[sommet,label=left:] (gamma) at (-5.5,-2) {} edge[thick] node[auto,swap,scale=0.7,below right] {$\infty$} (tau);
	\node[sommet,label=left:] (gamma) at (-5.5,-2) {} edge[thick] node[auto,swap,scale=0.7,below left] {$\infty$} (beta);
	\node[black] at (-7,-1.55) {$(\U_3)$};
	\node[sommet,label=below:] (a1) at (0,-1.25) {};
	\node[sommet,label=below:] (a2)  at (1,-1.25){}  edge[thick] node[auto,swap,scale=0.7] {$\infty$} (a1); 
	\node[sommet,label=left:] (a3) at (0,-2) {} edge[thick] node[auto,swap,scale=0.7,left] {$\infty$} (a1);
	\node[sommet,label=left:] (a4) at (1,-2) {} edge[thick] node[auto,swap,scale=0.7,below ] {$\infty$} (a3);
	\node[sommet,label=left:] (a4) at (1,-2) {} edge[thick] node[auto,swap,scale=0.7,right ] {$\infty$} (a2);
	\node[sommet,label=left:] (a3) at (0,-2) {} edge[thick] node[auto,swap,scale=0.7,right] {$\infty$} (a2);
	\node[sommet,label=below:] (a1)  at (0,-1.25){}  edge[thick] node[auto,swap,scale=0.7,left] {$\infty$} (a4); 
	\node[black] at (-1.2,-1.55) {$(\U_4)$};	
	\end{tikzpicture}
\end{center}

 The following proposition follows easily from classical facts in group theory regarding free groups and group presentations. 
 
\begin{prop}[Universal property of involution systems]
\label{prop:Universal} 
Let $(W,S)$ be an involution system of rank $n=|S|$ with Coxeter graph $\Gamma$. Consider the universal Coxeter system $(\mathcal U_n,S)$ and the Coxeter system $(W_\Gamma,S)$ of type $\Gamma$. Then there are unique group epimorphisms  $f:\U_n\to W$, $\varphi:W_\Gamma\to W$ and $\pi:\mathcal U_n\to W_\Gamma$ satisfying   $S=\varphi(S)=f(S)=\pi(S)$ such that the following diagram is commutative:
\[
\begin{tikzpicture}[>=latex]

\node (U) at (0,0) {$\mathcal U_n$};
\node (WG) at (0,-2) {$W_\Gamma$};
\node (W) at (3,0) {$W$};

% arrows
\draw[->>] (U) -- node[left] {$\pi$} (WG);
\draw[->>] (U) -- node[above] {$f$} (W);
\draw[->>] (WG) -- node[sloped,above] {$\varphi$} (W);

\end{tikzpicture}
\]
Moreover, $(W,S)$ is a Coxeter system if and only if $\varphi$ is an isomorphism.
\end{prop}
 
The following corollary follows directly from Proposition~\ref{prop:EIS-relations} and Proposition~\ref{prop:Universal}.

\begin{cor}
\label{cor:EIS-relations} 
Let $(W,S)$ be an involution system and $R_0\subseteq \mathcal U_S$ a set of relations in $\U_S$, i.e., $W=\mpair{S\mid \{s^2\mid s\in S\}\cup R_0}$. Then~$(W,S)$ is an \EIS if and only if $\ell_{(\mathcal U_S,S)}(w)$ is even for all $w\in R_0$. 
\end{cor}

%%%%%%
%%%%%%
\section{Weak order on involution systems}\label{se:WeakOrder}
Let $(W,S)$ be an involution system.  Recall that the {\em (right) Cayley graph\footnote{We consider Cayley graphs of group generated by involutions, so instead of considering an oriented graph having edges that are loop labeled by the same generator, we consider an undirected graph with edges labeled by the generators}  of $(W,S)$} is the edge-labeled graph $\Cay(W,S)$ whose vertices are the elements of $W$ and the edges are $\{w,ws\}$ with label $s$, where $w\in W$ and $s\in S$. The Cayley graph has a metric defined by:
$$
d_{(W,S)}(u,v)=\ell_{(W,S)}(u^{-1}v).
$$
As with the length, we denote $d:=d_{(W,S)}$ if there is no possible confusion. Moreover,  $W$ acts by isometries on $\Cay(W,S)$ by left-multiplication. It is also known that for any $w\in W$ the set of reduced words $\Red_w(W,S)$ is in bijection with the set of geodesics in $\Cay(W,S)$ from $\id$ to $w$; see for instance~\cite{Ep+92} for more details.

%%%%%%
\subsection{The weak order} 
\label{ss:WeakOrder}

We define the relation $\leq_R$ on $W$ by:
$$
u\leq_R w \iff u \text{ is a prefix of }w \iff u \textrm{ lies on  a geodesic from $\id$ to $w$}. 
$$

The following proposition follows easily from Proposition~\ref{prop:Length} and the definitions. 

\begin{prop} The relation $\leq_R$ is a partial order on $W$ with rank function $\ell:W\to \mathbb N^*$.
\end{prop}

\begin{defi}[Weak order] The graded poset $(W,\leq_R)$ is called the {\em right weak order} of $(W,S)$. Similarly, the {\em left weak order} $(W,\leq_L)$ is defined by $u\leq_L w$ if and only if $u$ is a suffix of $w$ (with the left Cayley graph). 
\end{defi}

The weak order on Coxeter systems is well-studied, see for instance~\cite[Chapter~3]{BjBr05}.

\begin{ex}[The absolute/reflection order on a Coxeter system is a weak order]
\label{ex:Absolute}
Let $(W,S)$ be a Coxeter system and $T=\bigcup_{w\in W}wSw^{-1}$ its set of reflections. We consider the \EIS $(W,T)$, also called the {\em dual presentation} of the Coxeter group $W$. The weak order associated to $(W,T)$ is called the {\em absolute order}.  The length function $\ell_{(W,T)}$ is called the {\em absolute length} or {\em reflection length} and was first studied by Carter, then Brady-Watts and Bessis studied it in relation to the {\em dual braid monoid}, see \cite{Be03,Arm09} for more details. 

In~\cite[\S4]{CaDySe23}, the authors defined {\em intermediate absolute orders} in relation to a family of subsets $T_k$ of $T$ ($k\in\mathbb N$) containing $S$: for each $k$, the $k$-absolute order is the weak order for the \EIS $(W,T_k)$. 
\end{ex}

More generally,  since the action of $W$ on $\Cay(W,S)$ by left-multiplication is by isometries, we define\footnote{The definition of the \emph{(right) weak order with respect to} $\omega$ is valid for any group $G$  with finite generating set $S$.} the {\em right weak order $(W, \leq_R^\omega)$ with respect to $\omega\in W$} by:
\begin{equation}\label{eq:Weak}
 g \leq_R^{\omega} h\iff\textrm{ $g$ lies on a geodesic in $\Cay(W,S)$ from $\omega$ to $h$.}
 \end{equation}
 In particular, the posets $(W, \leq_R^{\omega})$ and $(W, \leq_R^{\omega'})$ are isomorphic for any $\omega,\omega' \in W$. In this article, we only consider the right weak order so we simply say the {\em weak order} of $(W,S)$.

\smallskip 
The {\em semi-oriented Cayley graph $\oCay(W,S)$ of $(W,S)$} is the mixed graph obtained from $\Cay(W,S)$ by removing the label of the edges and orienting, if possible, $w$ to $ws$ if $\ell(ws)>\ell(w)$ (or equivalently if $\ell(ws)=\ell(w)+ 1$). See Figure~\ref{fig:A5} for an example. 

\begin{prop}\label{prop:EIS} Let $(W,S)$ be an involution system.
\begin{enumerate}
\item The Hasse diagram\footnote{We consider the Hasse diagram to be a directed graph with vertices $W$ and directed edges $(u,us)$ if $us$ cover $u$ in weak order.} of $(W,\leq_R)$ is obtained from $\oCay(W,S)$ by removing the undirected edges.
\item The Hasse diagram of $(W,\leq_R)$ is $\oCay(W,S)$ if and only if $(W,S)$ is an \EIS (even-involution system).
\end{enumerate}
\end{prop}
\begin{proof}  (1) follows from definitions and the fact that an edge is not directed if and only if $\ell(ws)=\ell(w)$. (2) If $(W,S)$ is an \EIS, then $\oCay(W,S)$ is a directed graph by Proposition~\ref{prop:Sign}. We conclude by (1). Conversely, if $\oCay(W,S)$ is a directed graph, then for any edge $(w,ws)$ we have by definition that $\ell(ws)=\ell(w)+1$, we conclude again by Proposition~\ref{prop:Sign}.
\end{proof}

Finally, the descent sets (Eq.~\eqref{eq:Descents}) of an involution system have a description in term of the weak order.

\begin{prop}
\label{prop:Descents}
 Let $(W,S)$ be an involution system, then $D_L=\{s\in S\mid s\leq_R w\}$ and $D_R=\{s\in S\mid s\leq_L w\}$.
 \end{prop}
 \begin{proof} Let $w\in W$. By definition of $\leq_L$, it is enough to show the equality for $D_L(w)$. Let $w\in W$, the inclusion $\{s\in S\mid s\leq_R w\}\subseteq D_L(w)$ follows from definition. Let $s\in S$ such that $\ell(sw)<\ell(w)$. Take $sw=s_1\cdots s_k$ a reduced word. Then $w=ss_1\cdots s_k$ is reduced since $k=\ell(sw)<\ell(w)=k+1$. So $s\leq_R w$.  
 \end{proof}

%%%%%%%%%%%%%%
\subsection{The weak order, lattice properties, \MIS and \EMIS}  In the case of Coxeter systems, Bj\"orner~\cite{Bj84,BjBr05}  showed that the weak order is a complete meet-semilattice.  The aim of this article is to study the question of which involution systems satisfy the property that $(W,\leq_R)$ is a complete meet-semilattice. First, we recall the notion of meet and join in $(W,\leq_R)$.  

Let $(W,S)$ be an involution system.  Given $\omega \in W$, we say that a non empty subset  $X$ of $W$ has a \emph{meet with respect to} $\omega$ if $X$ has a greatest lower bound $\bigwedge_R^{\omega} X$ with respect to $\leq_R^\omega$. When $\omega=\id$, we write $\bigwedge_R X$ and call this element the \emph{meet} of $X$. Since $(W, \leq_R^{\omega})$ and $(W, \leq_R)$ are isomorphic, we know that $\bigwedge_R^{\omega} X$ exists if and only if $\bigwedge_R X$ exists. Furthermore, given elements $x,y \in W$, we can explicitly write $x \wedge_R^{\omega} y = \omega(\omega^{-1}x \wedge_R \omega^{-1}y)$. 

The weak order of $(W,S)$ is a {\em complete meet-semilattice} if any non empty subset  $X$ of $W$ has a meet $\bigwedge_R X$. 

We say that a non empty subset $X$ is {\em bounded in $(W,\leq_R)$} if there is $\omega\in W$ such that $x\leq_R \omega$ for all $x\in X$. If the weak order is a complete meet-semilattice, then any bounded subset $X$ has a least upper bound $\bigvee_R X$ called the {\em join of $X$}:
$$
\bigvee_R X=\bigwedge_R\{u\in W\mid x\leq_R u,\,\forall x \in X\}.
$$

\smallskip

 We introduce some definitions.

\begin{defi}\label{def:Main} Let $(W,S)$ be an involution system. We say that:
\begin{enumerate}
\item $(W,S)$ is a {\em meet involution system}, or \MIS for short, if $(W,\leq_R)$ is a complete meet-semilattice.
\item $(W,S)$ is an {\em even meet involution system}, or \EMIS for short, if $(W,S)$ is both an \EIS and a \MIS.
\end{enumerate}
\end{defi}

\begin{rem}\label{rem:Complete} Because the set of reduced words for an element $w\in W$ is finite (since $S$ is finite), the argument at the end of the proof of~\cite[Theorem~3.2.1]{BjBr05} applies: in order to show that the weak order of an involution system $(W,S)$ is a complete meet-semilattice, it is enough to show that the meet of any two elements $x,y\in W$ exists.  
\end{rem}

\begin{ex} Coxeter systems are of course \EMIS by Bj\"orner's theorem.  But the dual presentation $(W,T)$, see Example~\ref{ex:Absolute}, {\em is not an \EMIS in general}: for finite Coxeter groups, each Coxeter element covers all reflections, so the meet does not exist for two Coxeter elements. However, intervals in the absolute order of finite Coxeter group are lattices, see for instance~\cite{Be03,BrWa08}.
\end{ex}

\begin{ex}  Consider the involution system $(W,S)$ given by the following presentation:
$$
W=\mpair{a,b,c\mid a^2=b^2=c^2=abc=\id}.
$$
Then $(W,S)$ is a \MIS but not an \EMIS and is not a Coxeter system: the Coxeter type $\Gamma$ and the Cayley graph are depicted below: 
\[
\begin{tikzpicture}[baseline=(mid.center),
  sommet/.style={inner sep=2pt,circle,draw=blue!75!black,fill=blue!40,thick}]
  \node[label=center:$(\Gamma)$]  at (-0.8,0) {};
  \node[sommet] (alpha) at (0,0) {};
  \node[sommet] (mid)   at (0.6,0) {};
  \node[sommet] (gamma) at (1.2,0) {};
\end{tikzpicture}
\qquad\quad\qquad
\begin{tikzpicture}[
    baseline=(c.center),
    scale=0.5,
    transform shape,
    vertex/.style={
        circle,
        draw,
        fill=blue!20,
        inner sep=1.6pt
    },
    elabel/.style={
        font=\small,
        fill=white,
        inner sep=1pt
    },
    blackedge/.style={
        draw=black,
        line width=0.4pt,
    },
    rededge/.style={
        draw=red,
        line width=0.4pt
    }
]

\node[vertex] (e) at (0,0) {$\id$};
\node[vertex] (a) at (-2,4) {$a$};
\node[vertex] (b) at ( 2,4) {$b$};
\node[vertex] (c) at ( 0,2.6) {$c$};

\draw[blackedge] (e) -- node[elabel, below left]  {$a$} (a);
\draw[blackedge] (e) -- node[elabel, below right] {$b$} (b);
\draw[blackedge] (e) -- node[elabel, right=0.1cm] {$c$} (c);

\draw[rededge] (a) -- node[elabel, above] {$c$} (b);
\draw[rededge] (a) -- node[elabel, below left]  {$b$} (c);
\draw[rededge] (b) -- node[elabel, below right] {$a$} (c);

\end{tikzpicture}
\]
The Hasse diagram of $(W,\leq_R)$ is obtained from $\Cay(W,S)$ by removing the edges not containing $\id$ (the red edges). 
\end{ex}

\begin{ex}\label{ex:EMIS-A2} The following involution system $(W,S)$, where
$$
W=\mpair{s_1,s_2,s_3\mid s_1^2=s_2^2=s_3^2=(s_1s_2s_3)^2=\id}
$$
is an \EMIS that is not a Coxeter system. 

Indeed, since $s_i^2$ and $(s_1s_2s_3)^2$, viewed as elements of $\U_S$, are of even length,  $(W,S)$ is an \EIS by Corollary~\ref{cor:EIS-relations}.  Moreover, $W$ acts on the left freely and transitively on the Honeycomb tiling (as depicted in Figure~\ref{fig:EMIS-A2}) so $\Cay(W,S)$ is the Honeycomb tiling. 
%The action is defined as follows. There are three equivalent classes of parallels for the edges of the hexagons in the Honeycomb tiling. Label each edge in the same equivalent class by the same generator, then for $i=1,2,3$, the generator $s_i$ exchanges the two vertices of all edges labeled by $s_i$ \marginpar{FABRICIO: Is this a right action? And what does it do to other types of edges}. 
The relation $(s_1s_2s_3)^2=\id$ is given by one of the hexagons containing $\id$. Observe that if $i\not = j$, the order $\ord_W(s_is_j)=\infty$. Therefore, the Coxeter graph of $(W,S)$ is of Coxeter type $\U_3$, but $W$ is a strict quotient of $\U_S$ and hence is not a Coxeter system by Proposition~\ref{prop:Universal}.

\begin{center}
\begin{figure}
\begin{tikzpicture}[scale=2]
  \def\s{0.5}
  \newcommand{\drawhexagon}[2]{%
    \foreach \k in {0,...,5} {%
      \coordinate (V\k) at ({#1+\s*cos(60*\k)},{#2+\s*sin(60*\k)});%
    }%
    \foreach \k in {0,...,5} {%
      \pgfmathtruncatemacro{\next}{mod(\k+1,6)}%
      \ifnum\k=0 \def\edgecolor{red}\fi%
      \ifnum\k=1 \def\edgecolor{blue}\fi%
      \ifnum\k=2 \def\edgecolor{green!60!black}\fi%
      \ifnum\k=3 \def\edgecolor{red}\fi%
      \ifnum\k=4 \def\edgecolor{blue}\fi%
      \ifnum\k=5 \def\edgecolor{green!60!black}\fi%
      \draw[thick, \edgecolor] (V\k) -- (V\next);%
    }%
  }

  % Parameters for the tiling: number of columns and rows
  \def\ncols{6} % 7 columns total
  \def\nrows{2}

  \pgfmathsetmacro{\vertSpacing}{\s*sqrt(3)}
  \pgfmathsetmacro{\horizSpacing}{1.5*\s}
  
  \foreach \col in {0,...,\ncols} {%
    \foreach \row in {0,...,\nrows} {%
      \pgfmathtruncatemacro{\modcol}{mod(\col,2)}
      \ifnum\modcol=1
         \pgfmathsetmacro{\yoffset}{\vertSpacing/2}
      \else
         \pgfmathsetmacro{\yoffset}{0}
      \fi
      \pgfmathsetmacro{\x}{\col*\horizSpacing}
      \pgfmathsetmacro{\y}{\row*\vertSpacing + \yoffset}
      \drawhexagon{\x}{\y}%
    }%
  }

 % labels
  \node at ({2.72*\horizSpacing},{0.87*\vertSpacing}) {\small $\id$};
  \node at ({2.53*\horizSpacing},{0.55*\vertSpacing}) {\small $s_1$};
  \node at ({2.5*\horizSpacing},{1.48*\vertSpacing}) {\small $s_2$};
  \node at ({3.29*\horizSpacing},{0.87*\vertSpacing}) {\small $s_3$};
  \node at ({2.42*\horizSpacing},{0*\vertSpacing}) {\small $s_1 s_2$};
  \node at ({3.7*\horizSpacing},{0*\vertSpacing}) {\small $s_1 s_2 s_3$};
 \node at ({3.4*\horizSpacing},{0.55*\vertSpacing}) {\small $s_3s_2$};
  
\end{tikzpicture}
\caption{The Cayley graph of the \EMIS of Example~\ref{ex:EMIS-A2}. The green edges are labeled by $s_1$, the red edges are labeled by $s_2$ and the blue edges are labeled by $s_3$. The Hasse diagram is obtained by orienting the Cayley graph by geodesic distance from $e$.}
\label{fig:EMIS-A2}
\end{figure}
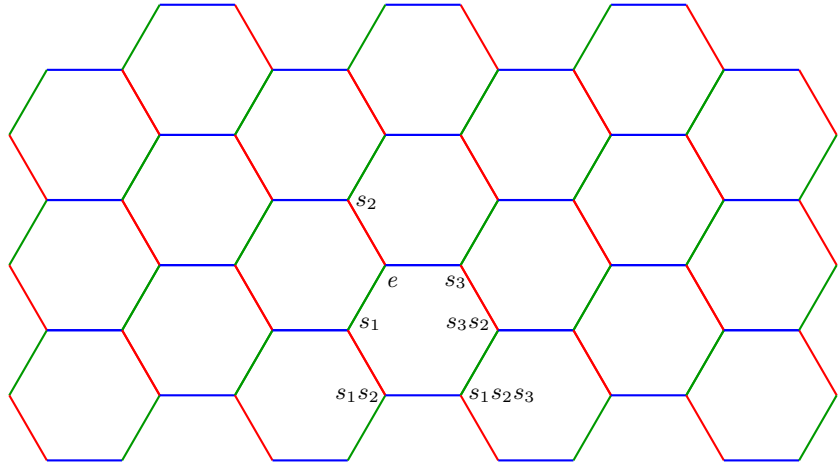
\end{center}
\end{ex}

In the above example, $\oCay(W,S)$ is isomorphic to the oriented Cayley graph of the affine Coxeter system of type $\tilde A_2$ by Proposition~\ref{prop:EIS} since $(W,S)$ is an \EIS; in particular the weak order of $(W,S)$ is lattice isomorphic to the one in the Coxeter system of type $\tilde A_2$. In \S\ref{Rank 3 section}, we show more generally that the weak order of any \EMIS of rank $3$ is lattice isomorphic to the weak order of a Coxeter system.  

More examples of \EMIS are discussed in \S\ref{Median subsection} below, \S\ref{Rank 3 section} and \S\ref{ss:Rank4}. We also describe in \S\ref{ss:Alt5} an example of a \MIS that is not an \EMIS.

%%%%%%%%
%%%%%%%
\subsection{\EIS with median Cayley graphs and cactus groups are \EMIS} \label{Median subsection} A graph is called \emph{median} if for any 3 vertices $x_1,x_2,x_3$ there exists an unique vertex $m$ such that $d(x_i,x_j) = d(x_i,m)+d(m,x_j)$ for all $i\neq j$. The vertex $m$ is called the \emph{median point} of $x_1,x_2,x_3$. 

Median graphs coincide with the one-skeletons of CAT(0) cube complexes (see \cite{MedianCAT(0)1, MedianCAT(0)2, MedianCAT(0)3}). Groups generated by involutions having a median Cayley graph are studied and classified in \cite{RAMRG} under the name of \emph{right-angled mock reflection groups}. Coxeter systems do not have median Cayely graphs but for trivial cases.

\begin{prop} \label{Median}
    Let $(W,S)$ be an \EIS whose Cayley graph is median. Then $(W,S)$ is an \EMIS. 
\end{prop}
\begin{proof} 
	By Remark~\ref{rem:Complete}, one just has to show the existence of the meet of two elements. Let $x,y \in W$. We show that the meet of $x$ and $y$ is the median point $m$ of $\id, x, y$. Let $z \in W$ be such that $z\leq_R x,y$ and consider the median point $n$ of $x,y,z$. Note that $z \leq_R n$. We then have that
	\begin{equation*}
        d(x,\id) =  d(x,z)+d(z,\id) = d(x,n) + d(n,z)+d(z,\id) = d(x,n)+ d(n,\id)
    \end{equation*}
    and similarly $d(y,\id) = d(y,n)+ d(n,\id)$. Thus, $n$ is also the median point of $x,y,\id$ which implies by uniqueness that $n=m$. Since $z \leq_R m \leq_R x,y$ it follows that $x \wedge_R y = m$. 
\end{proof}

The {\em cactus group} $J_n$ ($n\in\mathbb N_{\geq 2}$) associated to the symmetric group $\mathfrak S_n$  were introduced explicitly by Henriques and Kamnitzer~\cite{HeKa06} in the context of the category of crystals, and independently by Devadoss in \cite{Devadoss} in the context of operads. Since then, they have appeared in several important works. In~\cite{genevois2022cactus}, Genevois showed that  $J_n$ have median Cayley graphs.

More generally, Davis, Januszkiewicz, and Scott~\cite[\S4.7]{BlowUps} introduced {\em cactus groups over Coxeter systems}, see also~\cite{Go25} for more details.   In~\cite{Ge25}, Genevois showed that cactus groups over Coxeter systems have {\em quasi-median} Cayley graphs. However, a quasi-median graph is a median graph if it does not contains any triangle, i.e., $3$-cycles~\cite[Theorem~4.2]{Ge25}. By definition, a cactus group over a Coxeter system is an \EIS, so its Cayley graph does not contain $3$-cycles and therefore is a median graph. The following corollary follows from the above discussion and Proposition~\ref{Median}.

\begin{cor}
\label{cor:Cactus}
A cactus group over a Coxeter system is an \EMIS.
\qed
\end{cor}

%%%%%%%%%%%%
\subsection{The alternate group $\mathcal A_5$ is a \MIS but not an \EMIS}
\label{ss:Alt5}
\def\Ss{\mathfrak S}

The symmetric group $\Ss_5$ together with the set of simple transpositions $\tau_i=(i\ i+1)$ is a Coxeter system. 
The sign character $\varepsilon:\Ss_5\to\{\pm 1\}$ of the symmetric group $\Ss_5$ has for kernel  $\mathcal A_5=\ker\varepsilon$, the alternate group of order $60$.  Let $S=\{a,b,c\}$ with
$$
a=(12)(34),\ b=(23)(45)\ \text{and} \ c=(14)(23).
$$
Then $(\mathcal A_5,S)$ is an involution system of rank $3$ of Coxeter type
\begin{center}
 \begin{tikzpicture}[sommet/.style={inner sep=2pt,circle,draw=blue!75!black,fill=blue!40,thick}]
 \node[label=center:$H_3$:]  at (-1,0) {};
 	\node[sommet,label=below:$1$] (alpha) at (0,0) {};
	\node[sommet,label=below:$2$] (beta) at (1,0) {} edge[thick] node[above] {$5$} (alpha);
	\node[sommet,label=below:$3$] (gamma) at (2,0) {}edge[thick] node[above] {} (beta);
\end{tikzpicture}\hskip 1cm
 \end{center}
 which corresponds to the Coxeter system $(W',S')$ of type $H_3$, with $S'=\{s_1,s_2,s_3\}$. Therefore by Proposition~\ref{prop:Universal}, there is a unique epimorphism $\varphi: W'\to \mathcal A_5$ such that $\varphi(s_1)=a$,  $\varphi(s_2)=b$ and $\varphi(s_3)=c$.
 
\begin{center}
\begin{figure}[h!]
\includegraphics[width=12cm]{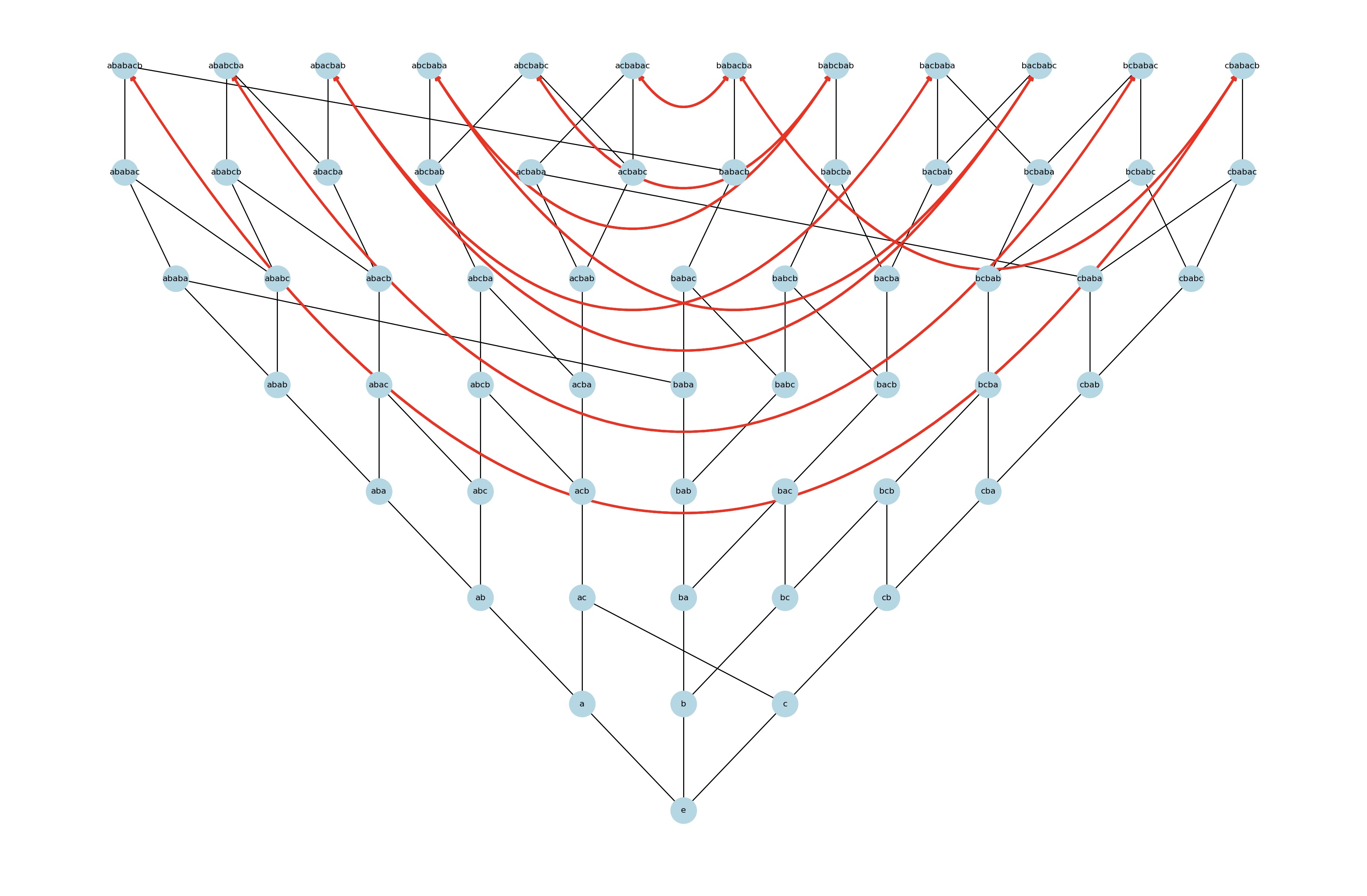}
\caption{The oriented Cayley graph $\oCay(\mathcal A_5,S)$. The red edges are those $\{w,ws\}$ that are not oriented, that is, $\ell(ws)=\ell(w)$. The Hasse diagram of $(\mathcal A_5,\leq_R)$ is obtained by considering only the black edges: it is a complete meet-semilattice but is not bounded.}
\label{fig:A5}
\end{figure}
\end{center}
 
However, $W$  is the isometry group of the dodecahedron of order $120$ and $\mathcal A_5$ is of order $60$.   Therefore, $(\mathcal A_5,S)$ is not a Coxeter system, since otherwise $\mathcal A_5$ would be isomorphic to $W$. 

Let $w_\circ=s_3s_2s_1s_2s_3(s_1s_2)^2s_3(s_1s_2)^2s_1$ the longest element of $(W,S)$. Then it is easy to check that $\ker\varphi=\mpair{w_\circ}$ is of order $2$ and therefore:
$$
\mathcal A_5=\mpair{S\mid a^2=b^2=c^2=(ab)^5=(bc)^3=(ac)^2=\ cbabc(ab)^2c(ab)^2a=e}.
$$

Now, consider $w=cbabcab=ababacba$ of length $7$. So $\ell(wa)=7=\ell(w)$, and therefore $(\mathcal A_5,S)$ does not have a sign character: $(\mathcal A_5,S)$ is not an \EIS. 

However, $(\mathcal A_5,\leq_{R})$ is a \MIS; see Figure~\ref{fig:A5}. Indeed, the Hasse diagram of ${{(\mathcal A_5,\leq_{R})}}$ can be seen as the subposet of $(W',\leq_R)$ considering only the elements of length smaller than or equal to $7$ (take the oriented subgraph of the Hasse diagram of $(W',\leq_{R})$ with vertices and edges arising from these elements). Note that the $\Cay(\mathcal A_5,S)$ can be seen as the folding  of $\Cay(W',S')$ mapping $w_\circ$ to $\id$, $w_\circ s$ to $s$ ($s\in S$), etc. Since $(W',\leq_R)$ is a complete lattice, $(\mathcal A_5,\leq_{R})$ is a complete meet-semilattice and therefore is a meet involution system (\MIS).

\begin{rem} This example can be generalized to the quotient of any finite Coxeter system $(W',S')$ by the subgroup of order $2$ generated by the longest element $w_\circ$ of $(W',S')$ if $w_\circ$ acts by conjugation on $S$ as the identity; this happens, for instance, in types $B$, $D_n$ ($n$ even), $E_7$, $E_8$, $F_4$, $H_3$ and $H_4$. It will be an \EMIS if $\ell(w_\circ)$ is even and just a \MIS otherwise (types $H_3$, $E_7$ and $B_n$ for $n$ odd). 
\end{rem}

%%%%%%%%%%%%%%%%%
%%%%%%%%%%%%%%%%%
%%%%%%%%%%%%%%%%%
%%%%%%%%%%%%%%%%%
\section{Presentation of \EMIS  by generators and relations}
\label{se:Presentation}
   In this section, we provide a presentation by generators and relations for any {\bf E}ven {\bf M}eet {\bf I}nvolution {\bf S}ystem. 

\smallskip
Let $(W,S)$ be an \EIS. The oriented Cayley graph $\oCay(W,S)$ is the Hasse diagram of $(W,\leq_R)$ by Proposition~\ref{prop:EIS}. Since $(W,S)$ is a quotient of the universal Coxeter system $(\U_S,S)$, there is a subset $R_0\subseteq\U_S$ such that 
$$
W=\mpair{S\mid \{s^2\mid s\in S\}\cup R_0},
$$
and any word $a_1\dots a_p\in R_0$ must have even length $p=2k$ in $\U_S$, by Corollary~\ref{cor:EIS-relations}. 
\smallskip

Now assume  $(W,S)$ is an  \EMIS and $s,t\in S$ bounded in $(W,\leq_R)$. Then there are two reduced words $sa_2\dots a_k$ and $tb_2\dots b_k$ such that $s\vee_Rt= sa_2\dots a_k=tb_2\dots b_k$ ($a_i,b_j\in S$).  Consider the element $\nu(s,t)\in \U_S$ defined by:
\begin{equation}
\label{eq:Relations}
\nu(s,t)=sa_2\dots a_kb_k\dots b_2t.
\end{equation}
The element $\nu(s,t)\in \U_S$ depends on the choice of the reduced words for $s\vee_R t$.

Note that if $(W,S)$ is of rank $2$, then $W$ is a dihedral group and thus $(W,S)$ is a Coxeter system generated by two involutions $s$ and $t$ with its classical Coxeter presentation and is  an \EMIS.

The aim of this section is to prove the following theorem.

\begin{thm}\label{thm:presentationEMIS} Let $(W,S)$ be an {\bf EMIS}.  Consider a total order $\prec$ on $S$. Let $R=\{s^2 \mid s\in S\}\cup R_0\subseteq \U_S$ where:
$$
R_0=\big\{\nu(s,t)\mid s\prec t \textrm{ in $S$ bounded in }(W,\leq_R)\big\}\subseteq \U_S.
$$
Then $W=\mpair{S\mid R}$ is a finite presentation of~$W$.
\end{thm}

In order to prove Theorem~\ref{thm:presentationEMIS}, we discuss in \S\ref{ss:CycleSpace} the notion of irreducible cycles in Cayley graphs. 

\begin{ex} Consider the \EMIS $(W,S)$ of Example~\ref{ex:EMIS-A2}. Then: 
$$
s_1\vee_R s_2= s_1 s_3 s_2=s_2s_3s_1,\ s_2\vee_R s_3= s_2 s_1 s_3=s_3s_1s_2,\ s_1\vee_R s_3= s_1 s_2 s_3=s_3s_2s_1. 
$$
In $\U_S$ we have therefore the following reduced words:
$$
\nu(s_1,s_2)=  (s_1 s_3 s_2)^2, \quad \nu(s_2, s_3)= (s_2 s_1 s_3)^2,\quad \nu(s_1, s_3)= (s_1 s_2 s_3)^2. 
$$
So $W=\mpair{s_1,s_2,s_3\mid s_i^2= (s_1 s_3 s_2)^2= (s_2 s_1 s_3)^2= (s_1 s_2 s_3)^2= \id}$ by Theorem~\ref{thm:presentationEMIS}.

\smallskip\noindent  Now, in $\U_S$, observe that $\nu(s_1,s_2)=s_2\nu(s_2,s_3)s_2$ and $\nu(s_1,s_3)=s_1\nu(s_1,s_2)^{-1}s_1$. Therefore the presentation given above is not minimal (relative to $S$), but the presentation $W=\mpair{s_1,s_2,s_3\mid s_i^2= (s_1 s_2 s_3)^2= \id}$ given in Example~\ref{ex:EMIS-A2} is minimal.
\end{ex}

\begin{rem} 
\label{rem:presentation}
\begin{enumerate}
\item As mentioned in the previous example, the presentation of $(W,S)$ in Theorem~\ref{thm:presentationEMIS} is not minimal in general; see~\S\ref{se:2recognizable} for further discussions.
\item Theorem~\ref{thm:presentationEMIS} does not apply to \MIS that do not have a sign character, i.e., that are not \EIS. For instance, for the \MIS $(\A_5,S)$ as in \S\ref{ss:Alt5}, the join of any pair $s,t$ of distinct generators provides the word $\nu(s,t)=(st)^{m(s,t)}$ in $\U_S$, where $m(s,t)=\ord_{\A_5}(st)$; however, the relation arising from the longest element of  the Coxeter system of type $H_3$ cannot appear in this way.
\item  The word $\nu(s,t)$ is reduced in $\U_S$. Indeed, assume by contradiction that  $\nu(s,t)$ is not reduced in $\U_S$. Then $a_k=b_k$ since $s$ and $t$ are distinct and the words $sa_2\dots a_k$ and $tb_2\dots b_k$ are reduced in $(W,S)$ so are in $\U_S$. But then  $s\vee_R t =sa_2\dots a_{k-1}=tb_2\dots b_{k-1}$, 
contradicting that $\ell(s\vee_R t)=k+1$.  
\end{enumerate}
\end{rem}

\begin{question} The presentation above is equivalent to the following {\em `braid-relation type'} presentation:
$$
W=\mpair{S\mid r^2=e, \ sa_2\dots a_k=tb_2\dots b_k ,\, \forall r \in S \text{ and }  s, t \in S \textrm{ bounded in }(W,\leq_R)}.
$$
By removing the relation $r^2=e$ ($r\in S$), we obtain the analog of an Artin-Tits group for an \EMIS. What are its properties? Is it the fundamental group of some interesting topological space?
\end{question}

We observe that this question has been studied by Scott in \cite{RAMRG} for the case of an \EMIS with median Cayley graph $(W,S)$. In particular, the author constructs for $W$ the analog of a Salvetti complex $Y$, and proves that the \emph{mock Artin group} $A$ corresponding to $W$ is the fundamental group of the quotient space $Y/W$, which is shown to be a finite $K(\pi,1)$ for $A$.

\smallskip
The following corollary provides a new characterization of Coxeter systems related to the notion of weak order.

\begin{cor}\label{cor:CoxeterCharac} Let $(W,S)$ be an involution system. Then the following conditions are equivalent.
\begin{enumerate}
\item\label{CoxCar-it1} $(W,S)$ is a Coxeter system;
\item\label{CoxCar-it2} $(W,S)$ is an \EMIS and for any $s,t\in S$ bounded, the reduced word $\nu(s,t)$ in $\U_S$ (defined in the above theorem) has its letters in $\{s,t\}$.
\end{enumerate}
\end{cor}
\begin{proof} The implication from (\ref{CoxCar-it1}) to (\ref{CoxCar-it2}) follows from the fact that for any $s,t\in S$ bounded in a Coxeter system,  $s\vee_R t$ is the longest element in the standard parabolic dihedral subgroup generated by $s$ and $t$, hence its reduced words have letters in $\{s,t\}$. Conversely, if~(\ref{CoxCar-it2}) holds, then the presentation in Theorem~\ref{thm:presentationEMIS} is one of a Coxeter system by definition. 
\end{proof}

%%%%%
\subsection{Irreducible cycles}\label{ss:CycleSpace}  A key notion to prove Theorem~\ref{thm:presentationEMIS} is the notion of irreducible cycles that we discuss now. 

Let $\G$ be a graph. Recall that a {\em cycle of $\G$} is a 2-regular connected subgraph of $\G$. A {\em circuit of $\G$} is the set of edges of a cycle of $\G$. The {\em length of a cycle} is its number of edges, or equivalently, its number of vertices. The \emph{edge space} $\mathcal{E}(\G)$ of $\G$ is the $\Z_2$-vector space consisting of all finite subsets of edges of $\G$ with symmetric difference as addition. The \emph{cycle space} $\mathcal{C}(\G)$ of $\G$ is the subspace of $\mathcal{E}(\G)$ consisting of the finite sums of circuits. Note that $\mathcal{C}(\G)$ coincides with the first singular homology group of $\G$ defined over $\Z_2$; see for instance \cite{cubicplanar} and the references therein for more information.

\begin{defi}
    We say that a cycle is {\em reducible} if its associated circuit is the sum in $\mathcal{C}(\G)$ of circuits of strictly lower length.  A cycle is called \emph{irreducible} if it is not reducible (see Figure~\ref{fig:exIrred} for a basic example). 
    
    Denote by $\Irr(\G)$ the set of irreducible cycles of $\G$. For a vertex $\omega$ of $\G$, we denote by $\Irr_\omega$ the set of irreducible cycles in $\Irr(\G)$ containing $\omega$ as a vertex. 
\end{defi}

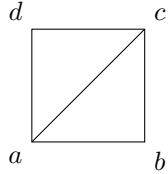
\begin{figure}
    \centering
\begin{tikzpicture}[scale=1.5]
  % vertices
  \coordinate (a) at (0,0);
  \coordinate (b) at (1,0);
  \coordinate (c) at (1,1);
  \coordinate (d) at (0,1);

  % square
  \draw (a) -- (b) -- (c) -- (d) -- cycle;

  % diagonal
  \draw (a) -- (c);

  % labels
  \node[below left]  at (a) {$a$};
  \node[below right] at (b) {$b$};
  \node[above right] at (c) {$c$};
  \node[above left]  at (d) {$d$};
\end{tikzpicture}
    \caption{A reducible cycle with two irreducible cycles: the cycle $(a,b,c,d)$ of length $4$ is not irreducible in $\G$ since it is reducible into the two cycles $(a,b,c)$ and $(a,c,d)$ of length $3$.}
    \label{fig:exIrred}
\end{figure}

In what follows, we consider $\G=\Cay(W,S)$ to be the Cayley graph of an \EMIS. The results in this section could be generalized to any graph having all cycles of even length: given a vertex $\omega$ of $\G$, the weak order $\leq_R^\omega$ is defined by the geodesic distance to $\omega$ as in \eqref{eq:Weak}; the complete meet-semilattice condition generalizes naturally. 

Since cycles and circuits are in one-to-one correspondence, we will use the two notions interchangeably. Moreover, given an edge-loop $\gamma$, we will, with slight abuse of notation, also denote by $\gamma$ the element of $\mathcal{C}(\G)$ determined by it (that is, the sum of circuits obtained by removing edges that appear an even number of times in $\gamma$). 

The aim of this section is to describe the structure of irreducible cycles in relation to the meet-semilattice property. We first prove the following lemma (which more generally holds for the Cayley graph of any \EIS).

\begin{lem} \label{lem-geo}
    Let $C \in \Irr(\G)$. Then $C$ has length $n=2k$ with $k\in\mathbb N_{\geq 2}$ and any edge-path of $C$ of length $ \leq k$ is a geodesic.
\end{lem}

\begin{proof} Let $\omega$ be a vertex in $C$. By reading the labels of the edges of $C$ in a loop from $\omega$ to $\omega$, we obtain a word $s_1\cdots s_n=\id$. Since $\G$ is the Cayley graph of an \EIS,  we have $n=2k$  by Corollary~\ref{cor:EIS-relations}.
    
    Now, suppose $\gamma$ is an edge-path of $C$ of length $p \leq k$ that is not a geodesic. Denote by $x,y$ the endpoints of $\gamma$, and let $l=d(x,y)<p$. Then we can decompose $C$ into a cycle of length $p+l<n$ and another cycle of length $n-p+l<n$, contradicting the fact that $C$ is irreducible. 
\end{proof}

We say that $x,\omega$ are {\em opposite vertices} in a cycle $C$ of length $2k$ in $\G$ if a geodesic in $C$ between $\omega$ and $x$ has length $k$. Recall also from \eqref{eq:Weak} the definition of the weak order related to an element  $\omega\in W$, i.e., a vertex of $\G$. The above lemma has the following useful consequence.

\begin{cor}
    Let $x,\omega$ be opposite vertices in a cycle $C \in \Irr(\G)$ and let $a,b$ be two other vertices of $C$ distinct from $x,\omega$. Then $a,b$ are comparable with respect to $\leq_R^{\omega}$ if and only if they are vertices of the same geodesic of $C$ with endpoints $x$ and $\omega$.
\end{cor}

Given an irreducible cycle $C \in \Irr(\G)$ and non-opposite vertices $a,b$ of $C$, we will denote by $\gamma^C_{ab}$ the unique geodesic of $C$ with endpoints $a$ and $b$. When the irreducible cycle in question is clear from the context, we will drop the superscript $C$ and just write $\gamma_{ab}$. Moreover, given edge-paths $\gamma$ and $\gamma'$ such that the endpoint of $\gamma$ is the starting vertex of $\gamma'$, we will write $\gamma\gamma'$ for the edge-path consisting of their concatenation.

\begin{lem} \label{irr-lemma}  Let $x,\omega$ be opposite vertices in a cycle $C \in \Irr(\G)$. Let $a,b$ be vertices of $C$ distinct from $x,\omega$ that are incomparable with respect to $\leq_R^{\omega}$. Let $y$ be a vertex of $\G$ satisfying $a,b \leq_R^{\omega} y$. Then:
$$
x=x \wedge_R^{\omega} y \leq_R^\omega y,\quad  a\vee_R^{\omega}b= x\quad \text{and}\quad a\wedge_R^{\omega}b= \omega.
$$
Moreover, if $d(x,\omega) = d(y,\omega)$ or if $\omega,a,b,y$ are all vertices in an irreducible cycle, then $y = x \wedge_R^{\omega} y = x$.
\end{lem}

\begin{proof}
    By Lemma \ref{lem-geo} we have that $a,b \leq_R^{\omega} x \wedge_R^{\omega} y \leq_R^{\omega} x,y$. Suppose that $x \wedge_R^{\omega} y \neq x$. Then $x \wedge_R^{\omega} y$ cannot be a vertex of $C$ since both $a \leq_R^{\omega} x \wedge_R^{\omega} y$ and $b \leq_R^{\omega} x \wedge_R^{\omega} y$. Let $\sigma$ be a geodesic with endpoints $x$ and $x \wedge_R^{\omega} y$, and let $\sigma_a$ and $\sigma_b$ be geodesics with endpoints $x \wedge_R^{\omega} y$ and $a$, and $x \wedge_R^{\omega} y$ and $b$, respectively. See Figure \ref{fig:irr-lemma} for an illustration. Then we can write $C = \gamma_{ax}\sigma\sigma_a + \gamma_{bx}\sigma\sigma_b + \sigma_a\gamma_{ab}\sigma_b$ where these three cycles have length strictly smaller than the length of $C$ since $x\neq x \wedge_R^{\omega} y \neq \omega$. This contradicts the fact that $C$ is irreducible, so we must have that $x = x \wedge_R^{\omega} y$, which proves  the two first statements. The third follows since $a,b$ are incomparable in $ \leq_R^{\omega} $. Now, if $\omega,a,b,y$ are all vertices of an irreducible cycle $C'$, then $y$ must be the opposite vertex to $\omega$ in $C'$ (since $a,b \leq_R^{\omega} y$), and the same proof gives $y = x \wedge_R^{\omega} y = x$. Alternatively, if $d(x,\omega) = d(y,\omega)$, then we also have $y= x \wedge_R^{\omega} y = x$ since $x =  x \wedge_R^{\omega} y \leq_R^{\omega} y$.
\end{proof}

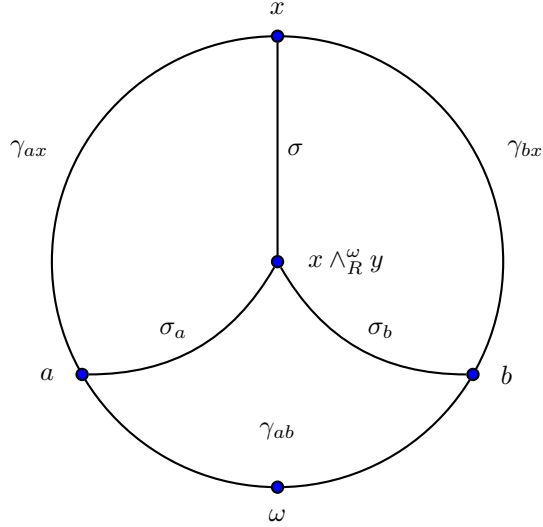
\begin{figure}[h!]
	\centering
	\begin{tikzpicture}[
		scale=3,
		vertex/.style={circle, draw=black, thick, fill=blue, inner sep=1.5pt}
		]
		
		\draw[thick] (0,0) circle (1);

		\node[vertex] (O) at (0,0) {};           % Central point
		\node[vertex] (X) at (0,1) {};           % Top vertex x
		\node[vertex] (A) at (-0.866,-0.5) {};   % Vertex a
		\node[vertex] (B) at (0.866,-0.5) {};    % Vertex b
		\node[vertex] (W) at (0,-1) {};          % Vertex omega

		\draw[thick] (O) -- (X) node[midway, right] {$\sigma$};
		\draw[thick] (O) to[bend left=30] node[midway, above left] {$\sigma_a$} (A);
		\draw[thick] (O) to[bend right=30] node[midway, above right] {$\sigma_b$} (B);
		
		% Labels for the vertices
		\node[above=5pt] at (X) {$x$};
		\node[left=7pt] at (A) {$a$};
		\node[right=7pt] at (B) {$b$};
		\node[below=5pt] at (W) {$\omega$};
		\node[right=8pt] at (O) {$x \wedge_R^\omega y$};
		
		% Labels for boundary arcs
		\node at (-1.1, 0.5) {$\gamma_{ax}$};
		\node at (1.1, 0.5) {$\gamma_{bx}$};
		\node at (0, -0.75) {$\gamma_{ab}$};
		
	\end{tikzpicture}
	\caption{Proof of Lemma \ref{irr-lemma}.}
	\label{fig:irr-lemma}
\end{figure}

\begin{defi}
 We say that a set $\mathcal{C}$ of cycles of $\G$ \emph{weakly intersects} if the intersection of any two distinct cycles in $\mathcal{C}$ does not contain two consecutive edges.
\end{defi}

The following proposition and its corollary are key in the proof of Theorem~\ref{thm:presentationEMIS}.

\begin{prop} \label{weakly-lemma} $\Irr(\G)$ weakly intersects. 
\end{prop}

\begin{proof}
    Suppose for a contradiction that two distinct cycles $C_1, C_2 \in \Irr(\G)$ share a common subpath of length two. Let $\omega$ be the middle vertex of this subpath, and denote by $a$ and $b$ the two other vertices of this subpath. Let $x$ and $y$ be the vertices opposite to $\omega$ in $C_1$ and $C_2$ respectively. Note that by Lemma \ref{lem-geo} we have $a,b \leq_R^{\omega} x \wedge_R^{\omega} y \leq_R^{\omega} x,y$, and $a,b$ are incomparable with respect to $\leq_R^{\omega}$. Therefore, by Lemma \ref{irr-lemma}, we have $x = x \wedge_R^{\omega} y = y$.

    Now, since $x=y$, we have that the geodesics $\gamma_{ax}^{C_1}, \gamma_{bx}^{C_1},\gamma_{ax}^{C_2}, \gamma_{bx}^{C_2}$ all have the same length $k$. For $l \in \{1,..,k-1\}$ let $\alpha$ be the vertex in $\gamma_{ax}^{C_1}$ at distance $l$ from $x$, and $\beta$ be the vertex in $\gamma_{ax}^{C_2}$ at distance $l$ from $x$. Let $\omega'$ be the vertex opposite to $\alpha$ in $C_1$. Then $x,b \leq_R^{\omega'} \alpha$ since $x,b$ are vertices in $C_1$. Suppose by contradiction that $\gamma^{C_2}_{\beta x}\gamma^{C_1}_{x\omega'}$ and $\gamma^{C_2}_{\beta b}\gamma^{C_1}_{b\omega'}$ are not geodesics. Then we have a geodesic $\sigma$ between $\beta$ and $\omega'$ of length $d <k+1$ that we use to decompose $C_1$ as $C_1 = \gamma^{C_2}_{\beta x}\gamma^{C_1}_{x\omega'}\sigma + \gamma^{C_1}_{\omega' b} \gamma^{C_2}_{b\beta}\sigma  +\gamma^{C_1}_{ax}\gamma^{C_2}_{ax}$. See Figure \ref{fig:weakly-lemma} for an illustration. But $\gamma^{C_2}_{\beta x}\gamma^{C_1}_{x\omega'}\sigma$ and $\gamma^{C_1}_{\omega' b} \gamma^{C_2}_{b\beta}\sigma$ have both length $k+1+d<2k+2$ and $\gamma^{C_1}_{ax}\gamma^{C_2}_{ax}$ has length $2k<2k+2$. This contradicts the irreducibility of $C_1$, thus $x,b \leq_R^{\omega'} \beta$. Hence, we can conclude by Lemma \ref{irr-lemma} that $\alpha = \beta$. Therefore, $\gamma^{C_1}_{ax} = \gamma^{C_2}_{ax}$, and by a similar argument, $\gamma^{C_1}_{bx} = \gamma^{C_2}_{bx}$. It then follows that $C_1=C_2$. 
\end{proof}

\begin{figure}[h!]
	\centering
	\begin{tikzpicture}[
		scale=3,
		vertex/.style={circle,draw=black, thick, fill=blue,inner sep=1.5pt},
		bluecurve/.style={blue,thick},
		redcurve/.style={red,thick}
		]
		
		% Main vertices
		\node[vertex,label=below:$a$] (a) at (-1,0) {};
		\node[vertex,label=above:$x$] (x) at (0,1.4) {};
		\node[vertex,label=below:$b$] (b) at (1,0) {};
		\node[vertex,label=below:$\omega$] (m) at ($(a)!0.5!(b)$) {};  % midpoint
		
		% Base
		\draw[thick] (a) -- (m) -- (b);
		
		% Outer symmetric red curve C2
		\draw[redcurve]
		(a) .. controls (-0.8,1.1) and (-0.25,1.25) .. (x)
		(x) .. controls (0.25,1.25) and (0.8,1.1) .. (b);
		
		% Inner symmetric blue curve C1
		\draw[bluecurve]
		(a) .. controls (-0.5,0.35) and (-0.15,0.95) .. (x)
		(x) .. controls (0.15,0.95) and (0.5,0.35) .. (b);
		
		% Alpha on left blue curve
		\path (a) .. controls (-0.5,0.35) and (-0.15,0.95) .. (x)
		node[pos=.55,vertex,label=right:$\alpha$] (alpha) {};
		
		% Beta on left red curve
		\path (a) .. controls (-0.8,1.1) and (-0.25,1.25) .. (x)
		node[pos=.48,vertex,label=left:$\beta$] (beta) {};
		
		% Omega' on right blue curve
		\path (x) .. controls (0.15,0.95) and (0.5,0.35) .. (b)
		node[pos=.65,vertex,label=right:$\omega'$] (omegap) {};
		
		% Sigma path
		\draw[thick] (beta) .. controls (-0.35,0.25) and (0.25,0.1) .. (omegap)
		node[pos=.45,below] {$\sigma$};
		
		% Curve labels
		\node[blue] at (0,0.82) {$C_1$};
		\node[red] at (0.65,1.05) {$C_2$};
		
	\end{tikzpicture}
	\caption{Proof of Lemma \ref{weakly-lemma}.}
	\label{fig:weakly-lemma}
\end{figure}

\begin{cor}\label{cor:IrrUnique} Let $\omega\in W$ and $s,t\in S$. If $C,C'\in \Irr_\omega$ contain $\omega s$ and $\omega t$ as vertices, then $C=C'$. In particular, $\Irr_\omega$ is finite. 
\end{cor}

\begin{proof} Assume $C,C'\in \Irr_\omega$ contain $\omega s$ and $\omega t$ as vertices. Then $\{\omega,\omega s\}$ and $\{\omega,\omega t\}$ are two consecutive edges in $C$ and $C'$, contradicting the fact that $\Irr(\G)$ weakly intersects (Proposition \ref{weakly-lemma}). Since $S$ is finite, $\Irr_\omega$ is finite.
\end{proof}

In fact, we prove in a corollary to the following proposition an even stronger statement about the intersection of irreducible cycles. A cycle $C$ is said to be {\em convex} if any geodesic between vertices $x$ and $y$ of $C$ in $\G$ is a geodesic in $C$.

\begin{prop}\label{prop:IrrConvex} Every irreducible cycle of $\G$ is convex.  
\end{prop}
\begin{proof} Let $C\in \Irr(\G)$ be of length $n=2k$.  Set $d=d(x,y)$. If $d=1$, the result follow since $\{x,y\}$ is then an edge of $C$. Assume now that $d>1$. Consider a geodesic $\gamma$ from $x$ to $y$ contained in $C$ and denote by $\sigma$ the path of length $n-d$ from $y$ to $x$ in $C$ such that $\gamma\sigma$ is the cycle $C$. Assume there is a geodesic $\gamma'$ from $x$ to $y$ that is not contained in $C$. Note that the cycles $\gamma'\gamma^{-1}$ and $\gamma'\sigma$ have length $\leq n$. Furthermore, $C$ has $d\geq 2$ consecutive edges in common with both $\gamma'\gamma^{-1}$ and $\gamma'\sigma$. Hence, neither $\gamma'\gamma^{-1}$ nor $\gamma'\sigma$ are irreducible by Proposition~\ref{weakly-lemma}. Therefore, we can obtain decompositions of both cycles into cycles of length strictly smaller than $n$. Combining these decompositions gives a decomposition of $C$ into cycles of strictly smaller length, a contradiction. 
\end{proof}

\begin{cor} \label{cor: stronger weak intersection}
	The intersection of any two distinct irreducible cycles of $\G$ is either empty, a single vertex or a single edge.
\end{cor}

\begin{proof}
	Let $C_1,C_2 \in \Irr(\G)$ and suppose $C_1 \cap C_2$ contains two nonadjacent vertices $x$ and $y$. Then, by Lemma \ref{lem-geo}, we can find a geodesic $\gamma_1$ in $C_1$ with endpoints $x$ and $y$. But by Proposition \ref{prop:IrrConvex}, $\gamma_1$ must belong to $C_2$ as well. Since it has length at least 2, we can conclude by Proposition \ref{weakly-lemma} that $C_1=C_2$.  
\end{proof}

%%%%%%%%%%%%%
%%%%%%%%%%%%%%
\subsection{Relations, irreducible cycles and the Coxeter companion of an \EMIS}
\label{ss:RelationsIrred} 
Let $(W,S)$ be an involution system with Cayley graph $\G$. For each irreducible cycle $C \in \Irr_{\id}$ we have:
\begin{itemize}
	\item The two edges adjacent to $\id$ are $\{e,s\}$ and $\{e,t\}$ for some $s,t\in S$ distinct. 
	\item There are exactly two loops starting and ending at $\id$: one by reading first the label $s$ on the edge $\{e,s\}$ and the other on by reading first label $t$ on the edge $\{e,t\}$. 
	\item Associated to those two loops, there are two reduced words $\nu_C$ and $\nu_C^{-1}$ in $\U_S$; one starting with $s$ and the other one starting with $t$. 
\end{itemize}
If $\omega$ is a vertex of $\G$ and $C \in\mathcal \Irr_\omega$, the words $\nu_C$, and the relation $\nu_C=\id$, are well-defined since the left action of $W$ is by isometries. Hence, any irreducible cycle of $\G$ is labeled by $\nu_C$ for some $C \in \Irr_{\id}$. Since $\Irr(\G)$ generates the cycle space $\mathcal{C}(\G)$, attaching 2-cells to all cycles labeled by the words $\nu_C$ ($C \in \Irr_e$) gives a simply-connected space (see \cite[Appendix \MakeUppercase{\romannumeral 1}.8A]{BridsonHaefliger} for the appropriate notions of 2-cell and attachment). Hence, we obtain the following proposition as a consequence of \cite[Lemma \MakeUppercase{\romannumeral 1}.8.9]{BridsonHaefliger} and Corollary \ref{cor:IrrUnique}.

\begin{prop}
\label{prop:RelationsIrred} 
A presentation of $(W,S)$ is given by the irreducible cycles having~$\id$ as a vertex: $W=\mpair{S \mid \{s^2 \mid s \in S\} \cup \{\nu_C \mid C \in \Irr_\id\}}$. In particular, if $(W,S)$ is an \EMIS then this presentation is finite.
\end{prop}

To an \EMIS $(W,S)$ we associate a Coxeter graph that encodes the size of its irreducible cycles. This notion will be useful in~\S\ref{Rank 3 section}.  Let $s,t \in S$ be distinct, then by Corollary~\ref{cor:IrrUnique} the edges $\{\id,s\}$ and $\{\id,t\}$ in $\Cay(W,S)$ uniquely determine an irreducible cycle of $\Irr_e$ if such a cycle exists. We assign to each such pair of generators a value $m(s,t) \in \N_{\geq 2}\cup \{\infty\}$ by setting $m(s,t)$ to be equal to half the length of the irreducible cycle of $\Irr_e$ containing the edges $\{\id,s\}$ and $\{\id,t\}$, or $\infty$ if no such cycle exists. 

\begin{defi}
\label{def:Companion} 
Let $(W,S)$ be an \EMIS. We consider the Coxeter matrix $m: \overline{S} \times \overline{S} \to \N \cup \{\infty\}$ defined by $m(s,t)$ for $s\neq t$ and $m(s,s) = 1$ for $s\in S$. The {\em companion graph of $(W,S)$} is the Coxeter graph $\Gamma$ associated to $m$. 

\noindent Let $\overline S=\{\overline s\mid s \in S\}$ be a copy of the alphabet $S$. The Coxeter system $(\overline{W}, \overline{S})$ of type $\Gamma$ is called \emph{Coxeter companion of} $(W,S)$.
\end{defi}

\begin{rem}
\begin{enumerate}
\item The Coxeter type of $(W,S)$ is not in general the companion graph of $(W,S)$. For instance, in Example~\ref{ex:EMIS-A2}, $(W,S)$ is of Coxeter type $\U_3$, the complete graph on three vertices with edges labeled by $\infty$, whereas its companion graph is the Coxeter graph of type $\tilde A_2$. 

\item If $(W,S)$ is a Coxeter system, then the companion graph is the Coxeter type of $(W,S)$. Indeed, the irreducible cycles in $\Irr_\id$ are the cycles corresponding to the finite dihedral standard parabolic subgroups $W_{s,t}$ for $s,t\in S$ with $m(s,t)<\infty$.

\item We show in Corollary~\ref{cor:CompanionCox} that $(W,S)$ is a Coxeter system if and only if its Coxeter type and companion graph are the same.
\end{enumerate}
\end{rem}

%%%%%%%%
%%%%%%%%
\subsection{Irreducible cycles, join of canonical generators and Proof of Theorem~\ref{thm:presentationEMIS}}
\label{ss:IrredJoin}

In this section, we show that for an \EMIS each irreducible cycle in $\Irr_{\id}$ is determined by a pair of canonical generators and their join. 

\begin{prop}\label{prop:EMISJoin} Let $(W,S)$ be an \EMIS and $C\in \Irr_e$. Then:
\begin{enumerate}
\item there are unique $s,t\in S$ that are vertices of $C$ and $s\vee_R t$ is the opposite vertex of $\id$ in $C$;

\item  there are $a_i,b_j\in S$, labels of edges of $C$, such that $\nu(s,t)=sa_2\dots a_kb_k\dots b_2t$ is a reduced word in $\U_S$;

\item The words $sa_2\dots a_k$ and $tb_2\dots b_k$ are reduced in $(W,S)$ and are the unique reduced words for $s\vee_R t$.
\end{enumerate}
\end{prop}
\begin{proof} (1) follows directly from Corollary~\ref{cor:IrrUnique} and Lemma~\ref{irr-lemma}. (2) let $\nu(s,t)=sa_2\dots a_kb_k\dots b_2t$ be the word in $\U_S$ obtained by reading the label of the edges of $C$ starting at the vertex $\id$ and going then to the vertex $s$ and finally to end at the vertex $t$. This word is reduced in $\U_S$ since there are no consecutive edge labeled with the same letter in $S$ in $\Cay(W,S)$. Finally, for (3): in $W$ we have the relation $sa_2\dots a_kb_k\dots b_2t=\id$, so $z= sa_2\dots a_k=tb_2\dots b_k$ satisfies $s\leq_R z$, $t\leq_R z$ and $z=s\vee_R t$ is the opposite vertex of $\id$ in $C$. In particular, the words $sa_2\dots a_k$ and $tb_2\dots b_k$ are reduced in $(W,S)$ since they correspond to geodesics from $\id$ to $z$ (by Lemma \ref{lem-geo}).  The unicity follows from Proposition~\ref{prop:IrrConvex}.
\end{proof}

\begin{proof}[{Proof of Theorem~\ref{thm:presentationEMIS}}]  Consider a total order $\prec$ on $S$. Let 
$$
R_0=\big\{\nu(s,t)\mid s\prec t \textrm{ in $S$ bounded in }(W,\leq_R)\big\}\subseteq \U_S,
$$
and $R=\{s^2 \mid s\in S\}\cup R_0$.  We have to show that $W=\mpair{S\mid R}$ is a finite presentation of $W$. By Proposition~\ref{prop:EMISJoin} and Proposition~\ref{prop:RelationsIrred}, the set
$$
R'_0=\big\{\nu(s,t)\mid s\prec t \textrm{ in $C$},\ C\in\Irr_\id\big\}
$$
gives the finite presentation $W=\mpair{S\mid \{s^2 \mid s\in S\} \cup R_0'}$. 
We have $R_0'\subseteq R_0$. Moreover, $\nu(s,t)=\id$ is already a relation for $(W,S)$, so 
$W=\mpair{S\mid \{s^2 \mid s\in S\} \cup R_0'}=\mpair{S\mid R}$.
\end{proof}

\begin{cor}\label{cor:CompanionCox} Let $(W,S)$ be an \EMIS. Then $(W,S)$ is a Coxeter system if and only if its Coxeter type and companion graph are the same. 
\end{cor}
\begin{proof} If $(W,S)$ is a Coxeter system, it follows from the general theory of Coxeter that the presentation $W=\mpair{S\mid (st)^{m(s,t)}=e \text{ for }s,t\in S\text{ with }m(s,t)<\infty }$ is minimal and that $\nu(s,t)=(st)^{m(s,t)}$, with the notation of Proposition~\ref{prop:EMISJoin}. So the irreducible cycles $\Irr_\id$ are the cycles corresponding to the finite dihedral standard parabolic subgroups $W_{s,t}$ for $s,t\in S$ with $m(s,t)<\infty$. Therefore the Coxeter type and the companion graph of $(W,S)$ are the same. 

Conversely, assume that the Coxeter type and the companion graph are the same. So for any $s,t\in S$ bounded, there is a cycle $C_{s,t}$ of length  $2\ord_W(st)$ containing $e$, $s$ and $t$ as vertices. We fix a total order $\prec$ on $S$. If $C_{s,t}\in \Irr_\id$, with $s\prec t$, then the word $\nu(s,t)$ in Theorem~\ref{thm:presentationEMIS} is $\nu(s,t)=(st)^{m(s,t)}$ and $W=\mpair{S\mid R}$ where $R=\{s^2\mid s\in S\}\cup R_0$ with
$$
R_0=\{(st)^{m(s,t)}\mid C_{s,t}\in\Irr_\id, \ s,t\textrm{ bounded} \}.
$$
So $(W,S)$ is a Coxeter system. Moreover, by the first implication, all $C_{s,t}$ are then in $\Irr_\id$.
\end{proof}

\begin{rem}
\label{rem:JoinIrred} Let $s,t\in S$ distincts and bounded such that $\ell(s\vee_R t)=k$. Write $z=s\vee_R  t$. If $k=2$, then $z=sa=tb$ for some $a,b\in S$ and the cycle $(\id, s,z,b)$ is irreducible by  Lemma~\ref{lem-geo}. 

Now assume $k>2$. Since $s,t\in D_L(z)$, we write $z=sa_2\dots a_k=tb_2\dots b_k$ reduced words.  Consider the cycle 
$$
C=(\id,s,sa_2,\cdots, sa_2\dots a_{k-1},z,b_{k-1}\dots b_1t,\cdots, b_1t,t).
$$
If $\{e,s\}$ and $\{e,t\}$ are edges in an irreducible cycle $C'$, then $C=C'$ by Proposition~\ref{prop:EMISJoin}. This lead us to ask the following question:  

\begin{question}\label{qu:Join} Does $C$ belong to $\Irr_\id$?  
\end{question}
If the answer to this question is positive, it would have the noticeable consequence that $D_L(s\vee_R t)=\{s,t\}$.  For cactus groups, the companion graph is made of squares and in this case, the answer to Question~\ref{qu:Join} is positive. 
\end{rem}

%%%%%%%%%%%%%%%%%%%%%%
%%%%%%%%%%%%%%%%%%%%%%
%%%%%%%%%%%%%%%%%%%%%%
%%%%%%%%%%%%%%%%%%%%%%
\section{$2$-recognizable presentations of \EIS}
\label{se:2recognizable}

Contrary to the Coxeter graph of a Coxeter system, the companion graph of an \EMIS is unfortunately not enough to fully encode the  finite presentation given in~Theorem~\ref{thm:presentationEMIS}. In this section we introduce the notion of $2$-recognizable presentation to describe the labels of the irreducible  cycles of an \EMIS. 

%%%%
\subsection{Words, cyclic words and factors}   
\label{ss:CyclicWords}

We consider the universal Coxeter system $(\U_S,S)$. The definition of a $2$-recognizable presentation is based on the notion of  cyclic words attached to  reduced words in $\U_S$. We recall first some well-known terminology on combinatorics of words. 

Recall that a {\em $p$-factor of a word $w\in S^*$} is a word $u$ of length $p\in\mathbb N^*$ such that $w=w_1uw_2$ for some words $w_1,w_2$ (possibly empty). A {\em letter $s\in S$ precedes $u$} if it is the last letter of the word $w_1$ (if $w_1$ is not empty); similarly, a {\em letter $s\in S$ follows $u$} if it is the first letter of the word $w_2$ (if $w_2$ is not empty).  In this article, we make use in particular of {\em $2$-factors}, i.e, factors $ab$ with $a,b\in S$. 

It is well-known that any element in $\U_S$ has a unique reduced word, that is, a word without $2$-factors $ss$ ($s\in S$). We identify therefore $\U_S$ with its set of reduced words:
$$
\Red=\Red(\U_S,S)=\{s_1\dots s_k\in S^*\mid s_i\not = s_{i+1}, \ \forall 1\leq i<k\}.
$$
Of course, if $w$ is a reduced word in $\U_S$ and $u$ is a $p$-factor of $w$, then  $w=w_1uw_2$ and  $u,w_1,w_2$ are reduced. 
 
 Let  $w=s_1\cdots s_k\in \U_S$  be a reduced word. 
\begin{itemize}
\item The {\em cyclic word  $[w]$} is the equivalence class of $w$ by cyclic shift:
$$
[w]=\{vu\mid u,v\in \Red,\ w=uv\}.
$$
\item We say that $[w]$ is a {\em reduced cyclic word} (or equivalently that {\em $w$ is cyclically reduced}) if any $g\in [w]$ is a reduced word in $\U_S$; or equivalently:  
\begin{equation}
\label{eq:CyclicRed}
[w]\quad \text{is a reduced cyclic word} \quad \iff\quad s_1\not = s_k.
\end{equation}
In particular, $w$ contains at least two distinct letters.

\item A word $u$ is a $p$-factor of $[w]$ if it is a $p$-factor of some $w' \in [w]$.  
%\item A {\em cyclic factor of $w$} is a reduced word $u\in \U_S$ that is a factor of some word in $[w]$. 
\end{itemize}

The following observation follows directly from definitions.

\begin{prop}
\label{prop:Factor}
Let $w=s_1\cdots s_k\in \U_S$ ($k\geq 2$) such that $[w]$ is a reduced cyclic word, then:
\begin{enumerate}
\item $[w^{-1}]$ is a reduced cyclic word;
 \item the words $s_kws_1$ and  $s_{k-1} s_kws_1 s_2$ are reduced in $\U_S$;
 \item For any $2$-factor $ab$ of $s_kws_1$, there is in  $s_{k-1} s_kws_1 s_2$ a letter $s$ that precedes $ab$ and a letter $r$ that follows $ab$ .
  \end{enumerate}
\end{prop}

%%%%
\subsection{$2$-recognizable involution systems}   
\label{ss:2-recog}

We are now ready to give the definition of a $2$-recognizable presentation.

\begin{defi}[$2$-recognizable subsets] \label{def: 2-rec subsets}
    We say that a subset $R_0 \subseteq \U_S$ is \emph{$2$-recognizable} if the following conditions hold:
\begin{enumerate}[series=2-recognizable]    
    \item \label{condition 1 - recognizable} $[w]$ is a reduced cyclic word for every $w \in R_0$;

    \item \label{condition 2 - recognizable} Every $w \in R_0$ has even length at least $4$;

	\item
	\label{condition 3 - recognizable}
	Let $w\in R_0$ and $ab$ a $2$-factor of $[w]$. If $ab$ is also $2$-factor of $[v]$ or $[v^{-1}]$ for some $v\in R_0$, then $w=v$.
   
    \item \label{condition 4 - recognizable} Let $w=s_1\cdots s_k\in R_0$ and $ab$ be a $2$-factor of $s_kws_1$. If $s\in S$ precedes (resp. follows) $ab$ in $s_{k-1} s_kws_1 s_2$, then for any occurrence of the factor $ab$ in $s_kws_1$,  $s$  also precedes (resp. follows) that occurence in $s_{k-1} s_kws_1 s_2$.
\end{enumerate}
\end{defi}

\begin{defi}[$2$-recognizable involution system]
Set $R= \{s^2 \mid s \in S\} \cup R_0$.  A presentation of an involution system $(W_R,S)$ of the form $ W_R=\mpair{S\mid R}$ for some 2-recognizable set $R_0$ is called a {\em $2$-recognizable presentation}, and $(W_R,S)$ is called a {\em $2$-recognizable involution system}. 
\end{defi}

\begin{rem}\label{rem:Recognizable}
	Condition~(\ref{condition 1 - recognizable}) implies in particular that any $w\in R_0$ does not start and end with the same letter in $S$. Condition~(\ref{condition 2 - recognizable}) implies that $(W_R,S)$ is an \EIS by Corollary~\ref{cor:EIS-relations}.
\end{rem}

Let $\cyc(w) = [w] \cup [w^{-1}]$, that is the set of all cyclic shifts of $w$ and $w^{-1}$. Moreover, let $\cyc(R_0) = \bigcup_{w \in R_0} \cyc(w)$. Condition (\ref{condition 3 - recognizable}) and (\ref{condition 4 - recognizable}) together are equivalent to the following:

\begin{enumerate}[resume=2-recognizable]
	\item \label{condition 5 - recognizable} For any $a,b \in S$ there exists at most one word in $\cyc(R_0)$ starting with $ab$. In particular, if $|S|<\infty$, then a $2$-recognizable presentation is finite. 
\end{enumerate}

Indeed, suppose (\ref{condition 5 - recognizable}) holds. Let $w,v \in R_0$ be distinct. Then $[w]$ or $[w^{-1}]$ cannot share a 2-factor $ab$ with $[v]$ or $[v^{-1}]$, otherwise there would be $w' \in \cyc(w)$ and $v' \in \cyc(v)$ both starting with $ab$. Hence, $(\ref{condition 3 - recognizable})$ holds. Similarly, if there was $w = s_1 \cdots s_k \in R_0$ and a 2-factor $ab$ of $s_kws_1$ such that $s$ precedes or follows an occurence of $ab$ in $s_{k-1} s_kws_1 s_2$ but not a different occurrence of $ab$ in $s_{k-1} s_kws_1 s_2$ that is an occurence in $s_kws_1$, then there would be $w_1, w_2 \in \cyc(w)$ such that $w_1 = abs \cdots$, $w_2 = abs' \cdots$ and $s \neq s'$, contradicting the assumption (\ref{condition 5 - recognizable}). Hence, (\ref{condition 4 - recognizable}) also holds.

Conversely, if $(\ref{condition 3 - recognizable})$ holds, then for any $w,v \in R_0$ distinct, no word in $\cyc(w)$ can start with the same two letters as a word in $\cyc(v)$. Furthermore, if $(\ref{condition 4 - recognizable})$ also holds, then given $w = s_1 \cdots s_k \in R_0$, since every occurence of a 2-factor $ab$ in $s_{k-1} s_kws_1 s_2$ that is also an occurence in $s_kws_1$ is followed and preceded by the same letter, we can conclude that there is a single word in $\cyc(w)$ starting with $ab$, so (\ref{condition 5 - recognizable}) holds.
 
 \begin{rem} \label{rem: small cancellation}
 	From the point of view of small cancellation theory (over free products) (see \cite[\S5.9]{LyndonScupp}), (\ref{condition 5 - recognizable}) means that every piece with respect to a 2-recognizable presentation has length 1.
 \end{rem}

Now, by (\ref{condition 1 - recognizable}) and (\ref{condition 2 - recognizable}), any word $w \in R_0$ determines an even length cycle in $\Cay(W_R,S)$ by reading the letters of $w$ starting at some vertex $g \in W_R$. Moreover, the set of all words labeling such cycles is $\cyc(R_0)$. Then, (\ref{condition 5 - recognizable}) guarantees that any two cycles associated to distinct words $w,w'$ of $R_0$ never have two consecutive edges in their intersection. This leads to the following result.

\begin{prop} \label{prop: 2-rec weakly intersects}
	Let $(W_R,S)$ be an involution system with $2$-recognizable presentation $\mpair{S\mid R}$ and let $\mathcal{C}$ be the set of cycles associated to words in $R_0$. Then $\mathcal{C}$ weakly intersects.
\end{prop}

Recall from \S\ref{ss:RelationsIrred} that associated to any irreducible cycle $C \in \Irr_{\id}$ in the Cayley graph of an \EMIS $(W,S)$ are two word $\nu_C$ and $\nu_C^{-1}$ obtained by reading the labels of edges of $C$ starting from $\id$. Furthermore, it is proved in Proposition \ref{prop:RelationsIrred} that $W=\mpair{S \mid \{s^2 \mid s \in S\} \cup \{\nu_C \mid C \in \Irr_\id\}}$ is a finite presentation of $W$. We now show that by removing redundant relations this gives a 2-recognizable presentation for $(W,S)$.

\begin{thm} \label{thm:EMIS is recognizable} Any \EMIS $(W,S)$ admits a finite $2$-recognizable presentation. In particular, the presentation obtained by considering only one word in $\cyc(\nu_C)$ for each $C \in \Irr_{\id}$ is $2$-recognizable.
\end{thm}

\begin{proof}
	For each $C \in \Irr_{e}$ choose one word in $\cyc(\nu_C)$ and let $R_0$ be the set of such words. Note that $R_0$ is finite by Proposition \ref{prop:RelationsIrred}. We claim that $R_0$ is a 2-recognizable set. Since each word $w \in R_0$ labels some cycle in $\Cay(W,S)$, $[w]$ must be a reduced cyclic word, so (\ref{condition 1 - recognizable}) is satisfied. Moreover, $(W,S)$ is an \EIS so (\ref{condition 2 - recognizable}) is also satisfied. In addition, since the irreducible cycles of $\Cay(W,S)$ weakly intersect by Proposition \ref{weakly-lemma}, for any distinct $a,b \in S$ there can be at most one word in $\cyc(R_0)$ starting with $ab$. Hence (\ref{condition 5 - recognizable}) is satisfied and thus $R_0$ is 2-recognizable. Finally, note that the normal closure of $\{\nu_C \mid C \in \Irr_{\id}\}$ equals the normal closure of $R_0$, so $\mpair{S\mid \{s^2 \mid s \in S\} \cup R_0}$ is also a presentation of $W$.  
\end{proof}

\subsection{Classification of $2$-recognizable presentations in rank 3} It turns out that the converse to Theorem \ref{thm:EMIS is recognizable} holds in rank 3. We prove this by classifying in this section all 2-recognizable presentations in rank 3. Then, in \S\ref{Rank 3 section}, we show that they are \EMIS.

\begin{thm}
\label{thm:rk3-reco}
    Let $(W,S)$ be an involution system of rank 3. Then $(W,S)$ is a $2$-recognizable involution system if and only if it admits one of the following presentations (where $\infty$ means there is no relation): 
    \begin{enumerate}[(i)]
        \item $\langle a,b,c \, | \, a^2=b^2=c^2 = (abc)^{2m}=\id\rangle$, $m \in \N_{\geq 1}$; \label{type1}
        
        \item $\langle a,b,c \, | \, a^2=b^2=c^2 = (abacbc)^{m}=\id \rangle$, $m \in \N_{\geq 1}$; \label{type2}
        
        \item $\langle a,b,c \, | \, a^2=b^2=c^2 = (abac)^m=(bc)^n=\id\rangle$, $m \in \N_{\geq 1}$, $n \in \N_{\geq 2} \cup \{\infty\}$; \label{type3}

        \item $\langle a,b,c \, | \, a^2=b^2=c^2 = (ab)^m=(bc)^n=(ac)^l=\id\rangle$, $m,n,l \in \N_{\geq 2} \cup \{\infty\}$ (the case of Coxeter systems of rank $3$). \label{type4}
    
    \end{enumerate}    
\end{thm}

\begin{proof}  It is readily seen that the presentations given above are $2$-recognizable. So we only need to classify all 2-recognizable presentations $W_R = \mpair{S \mid \{s^2 \mid s \in S\} \cup R_0}$ where $|S|=3$. We do this up to renaming letters or replacing an element $w \in R_0$ by $w' \in \cyc(w)$ since these operations do not change the group up to isomorphism.

    \smallskip
    Let $S = \{a,b,c\}$. Since there are only three letters in $S$, we must have that $|R_0| \leq 3$ by (\ref{condition 3 - recognizable}). If $R_0=\emptyset$, we get presentation \ref{type4} with $m=n=l=\infty$.

    \smallskip
    \noindent {\bf Case $|R_0| = 1$}.  Let $w = s_1 \cdots s_k \in R_0$ be nonempty. Without loss of generality we can assume $w$ starts with $ab$.
    
    Consider first the case when $xyx$ is not a subword of $w$, where $x,y \in S$. This forces the third letter of $w$ to be $c$, the fourth to be $a$, the fifth to be $b$, and so on. Since $w$ is cyclically reduced by (\ref{condition 1 - recognizable}) it cannot end with $a$. Furthermore, $w$ cannot end in $b$, otherwise $s_kws_1 = bwa$ would have an occurrence of $ab$ followed by $c$ and another followed by $a$, contradicting (\ref{condition 4 - recognizable}). Hence, since $w$ has even length at least 4 by (\ref{condition 2 - recognizable}), we must have that $w=(abc)^{2m}$ for some $m \in \N_{\geq 1}$ and we get presentation \ref{type1}. 
    
    Now, consider the case when $xyx$ is a subword of $w$ for some $x,y \in S$. Without loss of generality, assume $w$ starts with $aba$. 
    
    If the fourth letter of $w$ is $b$, then $w$ cannot contain the letter $c$. Otherwise, we could either find two occurrences of $ab$ in $s_kws_1$ where one is followed by $a$ and the other by $c$, or two occurrences of $ba$ in $s_kws_1$ where one is followed by $b$ and the other by $c$; contradicting (\ref{condition 4 - recognizable}). Since $w$ is cyclically reduced by (\ref{condition 1 - recognizable}) it cannot end with $a$, so in this case we get presentation \ref{type4} with $n=l=\infty$. 
    
    On the other hand, if the fourth letter of $w$ is $c$, then we have three cases. The first case is when $w$ has length 4, so $w=abac$ and we obtain the presentation \ref{type3} with $n=\infty$. 
    
    The second case is $|w|\geq 5$ and the fifth letter is $a$. We then claim that $w = (abac)^m$ for some $m\in \N_{\geq 1}$. Suppose not, that is $w = (abac)^mu$ for some $u \neq abac$ nonempty (note that if $m=1$ then $u$ starts with $a$). Let $s$ be the first letter of $u$ differing from $abac$, that is $u=u_1su_2$ where $u_1$ is a prefix of $abac$ but $abac \neq u_1su_2'$ for any word $u_2'$ (note that $s$ is possibly empty if $|u|<4$). Let $r_1,r_2$ be the two letters preceding $s$ in $w$. Then $r_1r_2$ is a 2-factor of $s_kws_1$ that has two occurrences in $s_kws_1$ followed by different letters, contradicting (\ref{condition 4 - recognizable}). Hence, we get presentation \ref{type3} with $n=\infty$. 
    
    Finally, the third case is $|w|\geq 5$ and the fifth letter is $b$. First, note that $w$ cannot be $abacb$ since $s_kws_1 = babacba$ has an occurrence of $ba$ followed by $b$ and another followed by $c$ and this would contradict (\ref{condition 4 - recognizable}). Hence, $|w|\geq 6$ and the sixth letter has to be $c$ since $w$ is cyclically reduced by (\ref{condition 1 - recognizable}). Then, a similar argument to the one in the previous paragraph gives us that $w = (abacbc)^m$ for some $m\in \N_{\geq 1}$, so we get presentation \ref{type2}.

    \smallskip
    \noindent {\bf Case $|R_0| = 2$}. Let $w_1,w_2 \in R_0$ be nonempty. If $w_1$ and $w_2$ only contain two different letters, then we get the presentation \ref{type4} with one of $m,n,l$ being infinity. Otherwise, assume that $w_1$ contains three different letters. From the case $|R_0|= 1$ we know that $w_1$ must be of the form $(abc)^{2m}$, $(abacbc)^m$ or $(abac)^m$ for some $m \in  \N_{\geq 1}$. But by (\ref{condition 3 - recognizable}), $w_1$ cannot be $(abc)^{2m}$ or  $(abacbc)^m$ since $xy$ or $yx$ is a 2-factor of $[(abc)^{2m}]$ and $[(abacbc)^m]$ for any $x,y \in S$. Hence, we must have $w_1 = (abac)^m$. Then, the only possibility for $w_2$ is $w_2 = (bc)^n$ where $n \in  \N_{\geq 2}$. Hence, we get presentation \ref{type3}.

    \smallskip
    \noindent {\bf Case $|R_0| = 3$}. Let $w_1,w_2,w_3 \in R_0$ be nonempty. If all of them only contain two different letters, then we get the presentation \ref{type4}. Otherwise, assume that $w_1$ contains three different letters. From the case $|R_0| = 2$, we must have that $w_1 = (abac)^m$ for some $m \in  \N_{\geq 1}$ and $w_2 = (bc)^n$ where $n \in  \N_{\geq 2}$. But if $w_3$ is nonempty, then we get a contradiction to (\ref{condition 3 - recognizable}) since $xy$ or $yx$ is a 2-factor of $[(abac)^m]$ or $[(bc)^n]$ for any $x,y \in S$. Hence, no word in $R_0$ can have three different letters, so the only possibility is presentation \ref{type4}.
  \end{proof}
  
  \begin{rem}     
     We observe that \ref{type1}, \ref{type2} and \ref{type3} (for $n \neq \infty$) appear as presentations (16), (17) and (18) in Georgakopoulos' classification of the planar cubic Cayley graphs \cite{cubicplanar}. The case $n = \infty$ in \ref{type3} corresponds to presentation (10) in \cite{cubicplanar}.
 \end{rem}

\subsection{A 2-recognizable involution system that is not an \EMIS} 
\label{sec: 2-rec not EMIS}
 Not all 2-recognizable involution systems are \EMIS. For instance, consider the involution system $(W,S)$ given by
 \begin{equation} \label{eq: 2-rec not EMIS}
 	W = \mpair{a,b,c,d \mid a^2=b^2=c^2=d^2=abcdcb=(ad)^2=(ac)^2=(bd)^2=e}.
 \end{equation}
 It is straightforward to verify that this is a 2-recognizable presentation. However, one can check, for example, that the elements $abc$ and $abab$ do not have a meet. We show part of the Cayley graph of $(W,S)$ in Figure \ref{fig:2-rec not EMIS}.

\begin{figure}[h!]
	\centering
	\begin{tikzpicture}[
		scale=2,
		vertex/.style={circle,draw=black,fill=blue, thick, inner sep=1.5pt},
		every label/.style={black,font=\small},
		edge/.style={line width=1.5pt}
		]
		
		% Radii chosen so octagon edges = square edges
		\def\rin{0.65}
		\def\rout{1.2}
		
		% Rotation of outer octagon to align with inner square
		\def\theta{-22.5}
		
		% Outer vertices
		\node[vertex,label=above:{$bab$}]  (v1) at ({135+\theta}:\rout) {};
		\node[vertex,label=above:{$abab$}] (v2) at ({90+\theta}:\rout) {};
		\node[vertex,label=right:{$aba$}]  (v3) at ({45+\theta}:\rout) {};
		\node[vertex,label=right:{$ab$}]   (v4) at ({0+\theta}:\rout) {};
		\node[vertex,label=below:{$a$}]    (v5) at ({-45+\theta}:\rout) {};
		\node[vertex,label=below:{$\id$}]  (v6) at ({-90+\theta}:\rout) {};
		\node[vertex,label=left:{$b$}]     (v7) at ({-135+\theta}:\rout) {};
		\node[vertex,label=left:{$ba$}]    (v8) at ({180+\theta}:\rout) {};
		
		% Inner vertices (square)
		\node[vertex,label=above:{$bac$}]  (u1) at (135:\rin) {};
		\node[vertex,label=above:{$abac$}] (u2) at (45:\rin) {};
		\node[vertex,label=below:{$abc$}]  (u3) at (-45:\rin) {};
		\node[vertex,label=below:{$bc$}]   (u4) at (-135:\rin) {};
		
		% Outer cycle
\draw[edge, red!70!black]   (v1)--(v2);
\draw[edge, blue]           (v2)--(v3);
\draw[edge, red!70!black]   (v3)--(v4);
\draw[edge, blue]           (v4)--(v5);
\draw[edge, red!70!black]   (v5)--(v6);
\draw[edge, blue]           (v6)--(v7);
\draw[edge, red!70!black]   (v7)--(v8);
\draw[edge, blue]           (v8)--(v1);

% Inner square
\draw[edge, orange!85!black] (u1)--(u2);
\draw[edge, blue]            (u2)--(u3);
\draw[edge, orange!85!black] (u3)--(u4);
\draw[edge, blue]            (u4)--(u1);

% Outer–inner connections
\draw[edge, green!60!black] (v8)--(u1);
\draw[edge, green!60!black] (v3)--(u2);
\draw[edge, green!60!black] (v4)--(u3);
\draw[edge, green!60!black] (v7)--(u4);
		
	\end{tikzpicture}
	\caption{Part of the Cayley graph of $(W,S)$ from the presentation in Eq.~(\ref{eq: 2-rec not EMIS}).}
	\label{fig:2-rec not EMIS}
\end{figure}
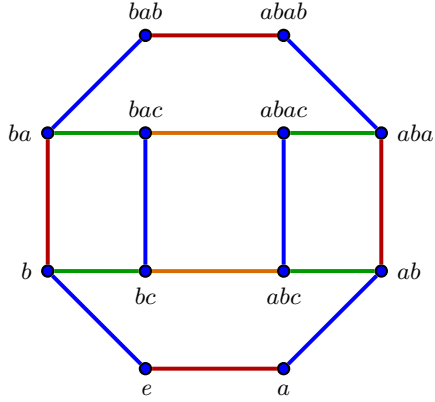

%%%%%%%%%%%%
%%%%%%%%%%%%
\section{Quasi-Coxeter systems and \EMIS of rank 3} 
\label{Rank 3 section}

Let $(W,S)$ be an \EIS. By Proposition~\ref{prop:EIS}, the Hasse diagram of $(W,\leq_R)$ is the oriented Cayley graph $\oCay(W,S)$. Example~\ref{ex:EMIS-A2} shows an \EMIS with  $\oCay(W,S)$ isomorphic, as a directed graph, to the Hasse diagram of a Coxeter system $(\overline W,\overline S)$. In particular, in this case, the semilattice structure is then entirely determined by the Coxeter companion of $(W,S)$. We first formalize this condition with a definition. 

\begin{defi}[Quasi-Coxeter systems] We say that $(W,S)$ is a {\em quasi-Coxeter system} if there is a Coxeter system $(\overline W,\overline S)$, such that $\oCay(W,S)$ and the Hasse diagram of $(\overline W,\overline S)$ are isomorphic (as directed graphs). In this case $(\overline W,\overline S)$ is the {\em Coxeter-companion of $(W,S)$} as defined in Definition~\ref{def:Companion}. 
\end{defi}

An example of a quasi-Coxeter system was already discussed in Example \ref{ex:EMIS-A2}. We also illustrate a hyperbolic quasi-Coxeter system in Figure~\ref{fig:Hypo}. It turns out that every \EMIS of rank 3 is a quasi-Coxeter system, as will be shown in \S\ref{ss: rank 3 EMIS}: see Theorem \ref{rank3-theorem} and Tables \ref{table: Classification in rank 3}~and~\ref{table2: Classification in rank 3} for a full classification.    However, this is not true for higher rank. In~\S\ref{se:FurtherOpen}, we provide an example of an \EMIS of rank $4$ that is not a quasi-Coxeter system. 

The following proposition, which motivates the definition, follows easily from definitions.

\begin{prop}\label{prop:QuasiCoxeter} Let $(W,S)$ be a quasi-Coxeter system with Coxeter companion $(\overline{W},\overline{S})$, then $(W,S)$ is an \EMIS and $(W,\leq_R)$ and $(\overline{W},\leq_R)$ are isomorphic as complete meet-semilattices.  \qed
\end{prop}

%%%%
\subsection{Growth series of Quasi-Coxeter systems and regularity of the language of reduced words}
\label{ss:Growth}

The {\em growth series} (or also called {\em Poincar\'e series}) of $(W,S)$ is the formal series in $q$ defined as follows:
$$
W(q)=\sum_{w\in W} q^{\ell(w)}.
$$
If $(W,S)$ is a Coxeter system, it is known that $W(q)$ is rational and a formula can be found in \cite[Chapter~7]{BjBr05}.

\begin{cor}[of Proposition~\ref{prop:QuasiCoxeter}]\label{cor:Growth} Let $(W,S)$ be a quasi-Coxeter system with Coxeter companion $(\overline{W}, \overline{S})$. Then $W(q)=\overline{W}(q)$. In particular, $W(q)$ is rational. 
\end{cor}

We show now that the language of reduced words of a quasi-Coxeter system is regular.  Recall that  the \emph{cone type} of $g \in W$ is the set
    \begin{align*}
        T(g) = \{h \in W : \ell(gh) = \ell(g) + \ell(h)\}.
    \end{align*}
    Equivalently, $T(g)$ is the set of elements belonging to a geodesic in $\Cay(W,S)$ through $g^{-1}$ and $\id$ which appear after $\id$; see for instance \cite{Ep+92}. Denote by $\mathbb{T}(W)$ the set of cone types of $W$. For Coxeter systems, it is known that $\T(W)$ is finite \cite[Corollary 2.10]{RegPartition}.

\begin{prop}
\label{prop:cone}
    If $(W,S)$ is a quasi-Coxeter system then $\T(W)$ is finite. 
\end{prop}

\begin{proof}
    Let $(\overline{W}, \overline{S})$ be the Coxeter companion of $(W,S)$ and $\varphi : W \to \overline{W}$ be the lattice isomorphism from Proposition \ref{prop:QuasiCoxeter}. Note that $\varphi$ maps any geodesic through $\id_W$ and $g^{-1}$ to a geodesic through $\id_{\overline{W}}$ and $\varphi(g^{-1})$. Hence, $\varphi$ induces a map $\Phi : \mathbb{T}(W) \to \mathbb{T}(\overline{W})$ given by $\Phi(T(g)) = T(\varphi(g^{-1})^{-1})$. Furthermore, $\Phi$ is injective since if $g,h \in W$ have different cone types, then so must $\varphi(g^{-1})^{-1}$ and $\varphi(h^{-1})^{-1}$. Hence, since $\mathbb{T}(\overline{W})$ is finite, it follows that $\T(W)$ is also finite.
\end{proof}

As a consequence of \cite[Theorem 1.2.9]{Ep+92} we obtain the following corollary.

\begin{cor}
\label{cor:regular}
If $(W,S)$ is a quasi-Coxeter system then $\Red(W,S)$ is regular. 
\end{cor}

\subsection{Quasi-Coxeter systems are biautomatic} \label{ss:Automatic} In the 1990's,  Davis--Shapiro~\cite{DS-au} and Brink--Howlett~\cite{BH-par} showed that every Coxeter group is automatic. In 2024, Osajda and Przytycki showed that every Coxeter group is in fact biautomatic~\cite{OP22}, and more biautomatic structures where exhibited in \cite{Sa25}. We refer the reader to \cite{Sa25} for the appropriate definitions of (bi)automatic structures and describe in this section how the biautomatic structure of Osajda and Przytycki on a Coxeter system can be ``transported'' to a quasi-Coxeter system with the same companion graph.
 
 \smallskip
 Let $(\overline{W}, \overline{S})$ be a Coxeter system. Recall that a \emph{reflection} is a conjugate of a generator in $\overline{S}$, and that associated to every reflection $r$ is a \emph{wall} $\W_r$, defined as the fixed point set of $r$ in $\Cay(\overline{W}, \overline{S})$. In particular, every wall separates the Cayley graph into two components, called \emph{half-spaces}. In \cite{OP22}, Osajda and Przytycki use the notion of walls and the weak order to construct a biautomatic structure on $(\overline{W}, \overline{S})$, namely the \emph{voracious language} $\overline{\V}$. Their construction goes as follows.
 
 Given $g \in \overline{W}$, let $\overline{\W}(g)$ be the set of walls separating $g$ from $\id$. They show that the set
 \[\overline{P}(g) = \{h \in \overline{W} \mid h \leq_R g \text{ and no wall in } \overline{\W}(g) \text{ separates } h \text{ from } \id\}\]
 has a unique largest element $\overline{p}(g)$ with respect to $(\overline{W},\leq_R)$. The \emph{voracious projection} 
 $$
 \overline{p}:\overline{W}\to \overline{W}
 $$
 is defined by  $g\mapsto \overline{p}(g)$. One of the properties of $\overline{p}$ that is relevant to us is the fact that for any isometry $\psi$ of the Cayley graph fixing the identity we have:
 $$
 \psi(\overline{p}(g)) = \overline{p}(\psi(g)).
 $$ 
 
 Osajda and Przytycki then define inductively on $\ell(g)$ a language $\overline{\mathcal{L}}_g$ of words over $\overline{S}$ representing $g$ by setting $\overline{\mathcal{L}}_{\id} = \{\varepsilon\}$ and
 \begin{align*}
 	\overline{\mathcal{L}}_g = \{uv : u \in \overline{\mathcal{L}}_{\overline{p}(g)}, v \in \overline{S}^* \text{ reduced word representing } \overline{p}(g)^{-1}g\}.
 \end{align*} 
 The \emph{voracious language} is then defined by
 \[\overline{\V} = \bigcup_{g \in \overline{W}} \overline{\mathcal{L}}_g,\]
 and Osajda and Przytycki show that $(\overline{\V} , \overline{S})$ is a biautomatic structure for $\overline{W}$.
 
 \smallskip
 Now, let $(W,S)$ be a quasi-Coxeter system with Coxeter companion $(\overline{W}, \overline{S})$ and $\varphi : \overline{W} \to W$ be a lattice isomorphism sending $\id_{\overline{W}}$ to $\id_W$. We define a \emph{wall} in $\Cay(W,S)$ to be $\varphi(\W)$ for some wall $\overline{\W} $ in $\Cay(\overline{W}, \overline{S})$. For an element $g \in W$, define then $\W(g)$ and $P(g)$ as in the case of $\overline{W}$; observe that $\W(g) = \varphi(\overline{\W}(\varphi^{-1}(g)))$ and $P(g) = \varphi(\overline{P}(\varphi^{-1}(g)))$. 
 
 We thus define the \emph{voracious projection $p:W\to W$} to be $p = \varphi\circ\overline{p}\circ\varphi^{-1}$, and note that $p(g)$ is the largest element of $P(g)$ with respect to $(W,\leq_R)$. Therefore, as in the case of $\overline{W}$, we are able to define inductively the language $\mathcal{L}_g$ for $g\in W$, and then  the \emph{voracious language $\V$ for $(W,S)$}, which satisfies  $\V=\varphi(\overline{\V})$.  We now explain why this language is regular and satisfies the fellow traveler property. 
 
 \subsubsection*{The fellow traveler property} Note that for any $\overline{g} \in \overline{W}$ we have $p(\varphi(\overline{g})) = \varphi(\overline{p}(\overline{g}))$. Hence, $\varphi$ maps $\overline{\LL}_{\overline{g}} \subset \overline{\V}$ to $\LL_{\varphi(\overline{g})} \subset \V$ bijectively. Now, let $\gamma$ be a path in $\Cay(\overline{W}, \overline{S})$ starting at a vertex $\overline{g}$ and labeled by a word in $\V$. Consider the isometry $f$ of $\Cay(\overline{W}, \overline{S})$ defined by $f(h) = \varphi^{-1}(\varphi(\overline{g})^{-1}\varphi(h))$ and let $t_{\overline{g}}$ denote translation by $\overline{g}$. Since $f \circ t_{\overline{g}}$ is an isometry fixing $\id_{\overline{W}}$, we have that $f(\gamma)$ is labeled by a word in $\overline{\V}$ and so $\varphi(f(\gamma))$ is labeled by a word in $\V$. Then, translating $\varphi(f(\gamma))$ by $\varphi(\overline{g})$, we obtain that $\varphi(\gamma)$ is labeled by a word in $\V$. A similar argument shows that $\varphi^{-1}$ maps any path in $\Cay(W,S)$ labeled by a word in $\V$ to a path in $\Cay(\overline{W}, \overline{S})$ labeled by a word in $\overline{\V}$. Hence, $\V$ satisfies the \emph{fellow traveler property} since $\overline{\V}$ does.
 
  \subsubsection*{Regularity of the voracious language}  Before proving regularity of $\V$, let us first introduce the appropriate notion of finite state automaton we work with. 
 
 \begin{defi}
 	A \emph{finite state automaton over S} (FSA for short) is a finite directed graph $\Gamma$ together with the following data:
 	\begin{enumerate}
 		\item a vertex set $A$ and an edge set $E \subseteq A \times A$;
 		\item a map $\phi : E \to \mathcal{P}(S^*)$ (the power set of $S^*$) assigning to each edge in $E$ a finite set of labels;
 		\item a vertex $a_0 \in A$ called the \emph{start state} of $\Gamma$;
 		\item a subset $A_{\infty} \subseteq A$ called the \emph{accept states} of $\Gamma$.
 	\end{enumerate}
 	The language $\LL(\Gamma)$ of the FSA $\Gamma$ is the set of words $v \in S^*$ which can be decomposed into subwords $v = v_0\ldots v_m$ labelling a directed edge-path $(e_0,\ldots,e_m)$ in $\Gamma$ from $a_0$ to a vertex in $A_{\infty}$, that is $v_i \in \phi(e_i)$ for $i=0,\ldots,m$. A subset $\mathcal{L} \subseteq S^*$ is a \emph{regular language} if it is the language of some FSA $\Gamma$ over $S$.
 \end{defi}
 
 In order to prove that $\V$ is a regular language, we can adapt the FSA $\overline{\A}$ that Osajda and Przytycki~\cite[Definition 6.2]{OP22} used to prove regularity of $\overline{\V}$. In particular, $\overline{\A}$ has as states the power set of the set of \emph{elementary walls} $\mathcal{E}$, that is, walls not separated from the identity by any other wall. Moreover, it has start state $\emptyset$ and every state is an accept state. The edges of $\overline{\A}$ from a state $A$ are then defined as follows. Suppose $\overline{w} \in \overline{W}$ satisfies:
 \begin{enumerate}
 	\item \label{FSA1} $\overline{w}$ is not separated from $\id_{\overline{W}}$ by any wall in $A$;
 	
 	\item \label{FSA2} $\overline{w}$ is separated from each wall in $A$ by another wall;
 	
 	\item \label{FSA3} $\overline{p}(\overline{w}) = \id_{\overline{W}}$.
 \end{enumerate}
 Then there is an edge $x$ from $A$ to the state $\overline{w}^{-1}\overline{\W}(\overline{w})$ with $\phi(x)$ consisting of all the reduced words representing $\overline{w}$.
 
 We adapt this to a FSA $\A$ as follows. The set of states of $\A$ is the power set of $\varphi(\mathcal{E})$. The start state is $\emptyset$ and all states are accept states. The transitions from a state $A$ are defined as follows. Suppose $w \in W$ satisfies:
 \begin{enumerate}
 	\item \label{FSA1-quasi} $w$ is not separated from $\id_W$ by any wall in $A$;
 	
 	\item \label{FSA2-quasi} $w$ is separated from each wall in $A$ by another wall;
 	
 	\item \label{FSA3-quasi} $p(w) = \id_W$.
 \end{enumerate}
 Then there is an edge $x$ from $A$ to the state $\varphi((\varphi^{-1}(w))^{-1}\overline{\W}(\varphi^{-1}(w)))$ with $\phi(x)$ consisting of all the reduced words representing $w$. The same proof as that of \cite[Proposition 6.1]{OP22} using the FSA $\A$ shows that $\V$ is regular. Since $\V$ satisfies also the fellow traveler property, we obtain the following result.

\begin{thm}
\label{thm:Automatic}
	Every quasi-Coxeter system is biautomatic. In particular, $(\V, S)$ is a biautomatic structure for $(W,S)$.
\end{thm}

\subsection{Rank $3$ \EMIS are quasi-Coxeter systems} \label{ss: rank 3 EMIS} The aim of this section is to prove the following theorem. 

\begin{thm} \label{rank3-theorem}
    Let $(W,S)$ be an involution system of rank $3$. Then the following statements are equivalent:
    \begin{enumerate}
     \item $(W,S)$ is an \EMIS; \label{rank3-theorem-1}
     \item $(W,S)$ admits a $2$-recognizable presentation; \label{rank3-theorem-2}
     \item $(W,S)$ admits one of the presentations in Theorem~\ref{thm:rk3-reco}; \label{rank3-theorem-3}
     \item \label{rank3-theorem-4} $(W,S)$ is a quasi-Coxeter system.
     \end{enumerate}
     In particular, every \EMIS of rank $3$ is a discrete subgroup of isometries of the sphere, Euclidean plane or hyperbolic plane.
\end{thm}

\begin{cor} \label{cor: companion rank 3}
	The companion graph of an \EMIS of rank $3$ completely determines its lattice structure.
\end{cor}

The classification of \EMIS of rank at most 3, with their companion graphs and their Coxeter graphs, appears in Tables~\ref{table: Classification in rank 3}~and~\ref{table2: Classification in rank 3}.

\begin{proof}[Proof of Theorem~\ref{rank3-theorem}]
Note that (\ref{rank3-theorem-1}) implies (\ref{rank3-theorem-2}) follows by Theorem \ref{thm:EMIS is recognizable}, (\ref{rank3-theorem-2}) implies (\ref{rank3-theorem-3}) follows by Theorem \ref{thm:rk3-reco} and (\ref{rank3-theorem-4}) implies (\ref{rank3-theorem-1}) follows from the definition of a quasi-Coxeter system. Hence, it remains to show that (\ref{rank3-theorem-3}) implies (\ref{rank3-theorem-4}), the proof of which is based on {\em Poincaré's Polyhedron Theorem} (see, for instance, \cite[\RomanNumeralCaps{4}.H]{Kleinian}).

Let $\Sb^2$ denote the sphere, $\E^2$ the Euclidean plane, and $\Hb^2$ the hyperbolic plane. To prove (\ref{rank3-theorem-3}) implies (\ref{rank3-theorem-4}), we consider $(W,S)$ with each of the 2-recognizable presentations \ref{type1}, \ref{type2} and \ref{type3} from Theorem \ref{thm:rk3-reco}. Then, we consider a triangle $\Delta$ in $\X = \Sb^2, \E^2$ or $\Hb^2$ with angles arising from the considered presentation. We then define an action by isometries on $\X$ for which $\Delta$ is the fundamental domain and such that the associated tiling of $\X$ is the same as the one obtained by the  action on $\Delta$ of $\overline W$, the discrete reflection group of rank $3$ associated to its Coxeter companion; the set $\overline S$ of simple reflections for $\overline W$ being constituted of the reflections across the sides of $\Delta$. Then, $\oCay(W,S)$ and $\oCay(\overline W,\overline S)$ will be the same (the orientation of edges being obtained by the geodesic distance from $\Delta$, which is labeled by $\id$). By definition, $(W,S)$ is therefore a quasi-Coxeter system.

\medskip

Consider first the case of presentation \ref{type1}. Let $\Delta$ be an equilateral triangle in $\X$ with vertices $A,B,C$ and internal angles $\frac{\pi}{3m}$ where $m \geq 1$. Note that $\Delta$ lies in $\E^2$ if $m = 1$ and in $\Hb^2$ otherwise. Let $a'$ denote the rotation by $\pi$ around the midpoint of $AB$, $b'$ denote the rotation by $\pi$ around the midpoint of $BC$ and $c'$ denote the rotation by $\pi$ around the midpoint of $AC$. In particular, these are \emph{side pairing transformations}, that is they map some edge of $\Delta$ to another edge of $\Delta$ (the same one in this case). 

Observe that $a'$ maps the edge $AB$ to itself and $A$ to  $a'(A) = B$. Moreover, $B$ belongs to exactly one other edge, $BC$, of $\Delta$, making an internal angle of $\alpha_1 = \frac{\pi}{3m}$ at $B$. Similarly, $b'$  maps $BC$ to itself and $B$ to the other vertex $g_2(B) = C$ of $BC$. Now, $C$ belongs to exactly one other edge, $AC$, of $\Delta$ making an internal angle of $\alpha_2 = \frac{\pi}{3m}$ at $C$.  We have again $c'(C) = A$ and $\alpha_3 = \frac{\pi}{3m}$. Hence, $c' \circ b' \circ a'(A) = A$ and $\alpha_1+\alpha_2+\alpha_3 = \frac{\pi}{m}$; the isometry  $h = c' \circ b' \circ a'$ is called a \emph{cycle transformation}. Since $AB$ makes an angle of $\frac{\pi}{m}$ with $h(AB)$, we need to apply this transformation $2m$ times to get back to $AB$. Hence, $h^{2m}(\Delta) = \Delta$ and we obtain the relation $(c'\circ b'\circ a')^{2m} = \text{Id}_{\X}$. We illustrate this in Figure \ref{fig: action type 1} for the case $m=1$. 
%Now, given a vertex $v_1$ of $\Delta$ and an edge $p_1$ of $\Delta$ containing $v_1$, we consider the rotation $g_1$ of angle $\pi$ with center the midpoint of $p_1$. In particular, $g_1$  is a side pairing transformation  mapping $p_1$ to itself and $v_1$ to the other vertex $g_1(v_1) = v_2$ of $p_1$. Moreover, $v_2$ belongs to exactly one other edge $p_2$ of $\Delta$, making an internal angle of $\alpha_1 = \frac{\pi}{3m}$ at $v_2$. Similarly, there is a side pairing transformation $g_2$ mapping $p_2$ to itself and $v_2$ to the other vertex $g_2(v_2) = v_3$ of $p_2$, which are rotations of angle $\pi$ with centers the midpoint of the other two edges. 

%Moreover, $v_3$ belongs to exactly one other edge $p_3$ of $\Delta$ making an internal angle of $\alpha_2 = \frac{\pi}{3m}$ at $v_3$. Note that $p_3$ is different from $p_1$ since $p_1$ does not contain $v_3$. Finally, doing this again yields $g_3(v_3) = X_1$, $p_4 = p_1$ and $\alpha_3 = \frac{\pi}{3m}$. Hence, $g_3 \circ g_2 \circ g_1(v_1) = v_1$ where the $g_i$ are distinct and $\alpha_1+\alpha_2+\alpha_3 = \frac{\pi}{m}$; the isometry  $h = g_3 \circ g_2 \circ g_1$ is called a \emph{cycle transformation}. Since $p_1$ makes an angle of $\frac{\pi}{m}$ with $h(p_1)$, we need to apply this transformation $2m$ times to get back to $p_1$. Hence, $h^{2m}(\Delta) = \Delta$ and we obtain the relation $(abc)^{2m} = \id$. We illustrate this in Figure \ref{fig: action type 1} for the case $m=1$.

\begin{figure}[h!]
	\centering
	\begin{tikzpicture}[scale=2]
		
		\def\r{1} % hexagon side length
		
		% --- automated coloring macro ---
		\newcommand{\hex}[2]{ 
			\begin{scope}[shift={#1}, rotate=#2]
				\foreach \i in {0,...,5} { \coordinate (V\i) at ({60*\i}:\r); }
				\coordinate (O) at (0,0);
				\foreach \i in {0,...,5} {
					\pgfmathtruncatemacro{\j}{mod(\i+1,6)}
					\pgfmathtruncatemacro{\rindex}{mod(\i + #2/60 + 12, 3)}
					\ifcase\rindex \def\rcol{red!70!black} \or \def\rcol{blue} \else \def\rcol{green!60!black} \fi
					\draw[thick, \rcol] (O) -- (V\i);
					\pgfmathtruncatemacro{\bindex}{mod(\i + 2 + #2/60 + 12, 3)}
					\ifcase\bindex \def\bcol{red!70!black} \or \def\bcol{blue} \else \def\bcol{green!60!black} \fi
					\draw[thick, \bcol] (V\i) -- (V\j);
				}
			\end{scope}
		}
		
		% Render the base hexagons
		\hex{(0,0)}{0}
		\hex{(1,0)}{60}
		\hex{(0.5,{sqrt(3)/2})}{-60}
		
		% Define coordinates
		\coordinate (A) at (0,0);
		\coordinate (B) at (1,0);
		\coordinate (C) at (0.5,{sqrt(3)/2});
		
		\coordinate (Mab) at ($(A)!0.5!(B)$); % Horizontal (Red)
		\coordinate (Mbc) at ($(B)!0.5!(C)$); % -60 deg (Green)
		\coordinate (Mca) at ($(C)!0.5!(A)$); % +60 deg (Blue)

		\tikzset{
			rotarrow/.style={
				{Stealth[length=3pt, bend]}-, 
				>=Stealth, 
				thin
			}
		}
		
		% Point a (Bottom, Horizontal Line)
		% Gap centered at 270 deg
		\fill (Mab) circle (0.5pt);
		\draw[rotarrow] ($(Mab)+(250:0.12)$) arc (250:-70:0.12); 
		\node[below=0.4] at (Mab) {$a$};
		
		% Point b (Right, Green Line at -60 deg)
		% Gap centered at 30 deg (perpendicular to -60)
		\fill (Mbc) circle (0.5pt);
		\draw[rotarrow] ($(Mbc)+(10:0.12)$) arc (10:-310:0.12);
		\node[right=0.4] at (Mbc) {$b$};
		
		% Point c (Left, Blue Line at +60 deg)
		% Gap centered at 150 deg (perpendicular to 60)
		\fill (Mca) circle (0.5pt);
		\draw[rotarrow] ($(Mca)+(130:0.12)$) arc (130:-190:0.12);
		\node[left=0.4] at (Mca) {$c$};
		
	\end{tikzpicture}
	\caption{The action of $W$ for $m=1$ in presentation \ref{type1}.}
	\label{fig: action type 1}
\end{figure}
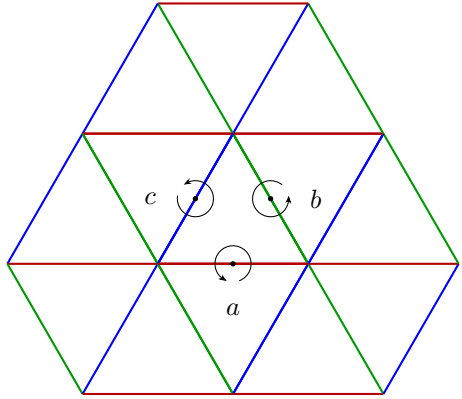

We obtain therefore a surjective morphism of groups from $W$ to $W'=\langle a',b',c'\rangle$ mapping $a$ to $a'$, $b$ to $b'$ and $c$ to $c'$. By Poincaré's Polyhedron Theorem, the group $W'$ acts on a locally finite tiling $\mathcal{T}$ of $\X$ by translates of $\Delta$, and has presentation \ref{type1}. In particular, $W$ is isomorphic to $W'$, $W$ acts on  $\mathcal{T}$ transitively, and the generators in $S$ map a triangle to its adjacent triangles. Hence, the Cayley graph of $(W,S)$ is dual to the tiling $\mathcal{T}$, so it is isomorphic (as an unlabeled graph) to the Cayley graph of the triangle group $(3m,3m,3m)$, the Coxeter companion of $(W,S)$. See, for instance, Figure \ref{fig:EMIS-A2} for the Cayley graph of $(W,S)$ when $m=1$. Therefore, $(W,S)$ is a quasi-Coxeter system.  

\medskip

Now, we consider the case of presentation \ref{type2}. Again, let $\Delta$ be an equilateral triangle in $\X$ with vertices $A,B,C$ and internal angles $\frac{\pi}{3m}$ where $m \geq 1$. Let $a'$ denote the rotation by $\pi$ around the midpoint of $AB$, $b'$ denote the reflection across the side $BC$ and $c'$ denote the rotation by $\pi$ around the midpoint of $AC$.

%The difference in this case is that the side pairing transformation $b'$ associated to the edge $BC$ fixes both endpoints $B$ and $C$. Hence, applying the same construction as before starting with a vertex $A$ and edge $AB$ of $\Delta$, we obtain $g_6 \circ g_5 \circ g_4 \circ g_3 \circ g_2 \circ g_1(X_1) = X_1$ where $g_1 = g_3 \neq g_4 = g_6$ are distinct from $g_2=g_5$ if $p_1 \neq BC$, or $g_1=g_4$ distinct from $g_2=g_6\neq g_3=g_5$ if $p_1 = BC$. Moreover, $\alpha_1 + \ldots + \alpha_6 = \frac{2\pi}{m}$. Hence, $(g_6 \circ \ldots \circ g_1)^m(\Delta) = \Delta$ and we obtain the relation $(abacbc)^m = \id$. We illustrate this in Figure \ref{fig: action type 2} for the case $m=1$.

The difference in this case is that the side pairing transformation $b'$ associated to the edge $BC$ fixes both endpoints $B$ and $C$. Hence, applying the same construction as before starting with a vertex $A$ and edge $AB$ of $\Delta$, we obtain $h(A)=c' \circ b' \circ c' \circ a' \circ b' \circ a'(A) = A$. Moreover, the angle between $h(AB)$ and $AB$ is $6\cdot\frac{\pi}{3m} = \frac{2\pi}{m}$. Hence, $h^m(\Delta) = \Delta$. As in the first case, we conclude by Poincar\'e's Polyhedron Theorem that there is an isomorphism between $W$ and  $W'=\langle a',b',c'\rangle$, and that $W$ acts on the associated tiling $\mathcal{T}$ transitively, where generators map a triangle to its adjacent triangles. Hence, $(W,S)$ is also a quasi-Coxeter system. We illustrate this in Figure \ref{fig: action type 2} for the case $m=1$.

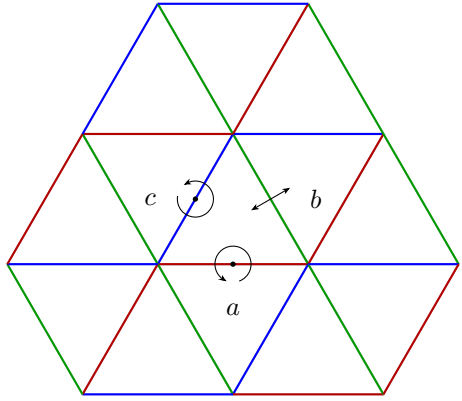
\begin{figure}[h!] 
	\centering 
	\begin{tikzpicture}[scale=2] 
		
		\def\r{1} 
		\def\h{0.8660254} % sqrt(3)/2
		
		% --- Base Hexagon Logic ---
		\newcommand{\hex}[2]{ 
			\begin{scope}[shift={#1}, rotate=#2] 
				\foreach \i in {0,...,5} { 
					\coordinate (V\i) at ({60*\i}:\r); 
				} 
				\coordinate (O) at (0,0); 
				\foreach \i in {0,...,5} { 
					\pgfmathtruncatemacro{\j}{mod(\i+1,6)} 
					\draw[thick, gray!30] (O) -- (V\i) -- (V\j) -- cycle; 
				} 
			\end{scope} 
		} 
		
		% Draw the ghost hexagons
		\hex{(0,0)}{0} 
		\hex{(1,0)}{60} 
		\hex{(0.5,\h)}{-60} 
		
		% --- Defined Vertices --- 
		\coordinate (A) at (0,0); 
		\coordinate (B) at (1,0); 
		\coordinate (C) at (0.5,\h); 
		
		% --- Remaining Triangle Vertices --- 
		\coordinate (C1) at (0.5,-\h);   % Bottom
		\coordinate (D1) at (1.5,\h);    % Right
		\coordinate (E1) at (-0.5,\h);   % Left
		\coordinate (F1) at (1.5,-\h);
		\coordinate (G1) at (-0.5,-\h);
		\coordinate (H1) at (-1,0);
		\coordinate (I1) at (2,0);
		\coordinate (J1) at (0,2*\h);
		\coordinate (K1) at (1,2*\h);
		
		% --- Draw the segments ---
		\draw[red!70!black, thick]   (A) -- (B); 
		\draw[green!60!black, thick] (B) -- (C); 
		\draw[blue, thick]  (C) -- (A); 
		
		% Bottom Triangle
		\draw[blue, thick]  (B) -- (C1); 
		\draw[green!60!black, thick] (C1) -- (A); 
		
		% Right Triangle 
		\draw[blue, thick]   (D1) -- (C);    
		\draw[red!70!black,  thick] (B) -- (D1); 
		
		% Left Triangle
		\draw[red!70!black, thick]   (E1) -- (C);    
		\draw[green!60!black, thick] (A) -- (E1); 
		
		\draw[red!70!black, thick]   (C1) -- (F1);    
		\draw[green!60!black, thick] (B) -- (F1); 
		
		\draw[blue, thick]   (C1) -- (G1);    
		\draw[red!70!black,  thick] (A) -- (G1);
		
		\draw[green!60!black, thick]   (H1) -- (G1);    
		\draw[blue, thick] (A) -- (H1); 
		
		\draw[red!70!black,  thick]   (H1) -- (E1);
		
		\draw[blue, thick]   (E1) -- (J1);    
		\draw[green!60!black, thick] (C) -- (J1);
		
		\draw[green!60!black, thick]   (D1) -- (K1);    
		\draw[red!70!black, thick] (C) -- (K1); 
		
		\draw[blue, thick]   (J1) -- (K1);
		
		\draw[green!60!black, thick]   (D1) -- (I1);    
		\draw[blue, thick] (B) -- (I1); 
		
		\draw[red!70!black, thick]   (I1) -- (F1);
		
		\tikzset{
			rotarrow/.style={
				{Stealth[length=3pt, bend]}-, 
				>=Stealth, 
				thin
			}
		}
		
		\tikzset{
			refarrow/.style={
				<->,               
				{Stealth[length=3pt, bend]}-{Stealth[length=3pt, bend]}, 
				>=Stealth,          
				thin
			}
		}
		
		% Point a (Bottom, Horizontal Line)
		% Gap centered at 270 deg
		\fill (Mab) circle (0.5pt);
		\draw[rotarrow] ($(Mab)+(250:0.12)$) arc (250:-70:0.12); 
		\node[below=0.4] at (Mab) {$a$};
		
		% Point b (Reflection: Double-headed arrow perpendicular to BC)
		\coordinate (Mbc) at ($(B)!0.5!(C)$); 
		\draw[refarrow] ($(Mbc)+(210:0.15)$) -- ($(Mbc)+(30:0.15)$); 
		\node[right=0.4] at (Mbc) {$b$};
		
		% Point c (Left, Blue Line at +60 deg)
		% Gap centered at 150 deg (perpendicular to 60)
		\fill (Mca) circle (0.5pt);
		\draw[rotarrow] ($(Mca)+(130:0.12)$) arc (130:-190:0.12);
		\node[left=0.4] at (Mca) {$c$};
		
	\end{tikzpicture} 
	\caption{The action of $W$ for $m=1$ in presentation \ref{type2}.}
	\label{fig: action type 2}
\end{figure}

%We can then conclude, by Poincaré's Polyhedron Theorem again, that the group $W$ presented by \ref{type2} has Cayley graph dual to the tiling $\mathcal{T}$ of $\X$ by translates of $\Delta$. Hence, $(W,S)$ is also a quasi-Coxeter system.  

\medskip

Finally, we analyze presentation \ref{type3}. Let $\Delta$ be an isosceles triangle in $\X$ with vertices $A,B,C$ and internal angles $\frac{\pi}{2m}$ at $A$ and $B$, and $\frac{\pi}{n}$ at $C$, where $m \geq 1$ and $n \geq 2$ or $\infty$. For $n\neq \infty$, note that $\Delta$ lies in $\Sb^2$ if $(m,n) = (1,n)$, in $\E^2$ if $(m,n) = (2,2)$, and in $\Hb^2$ otherwise. For $n=\infty$ and $m \neq 1$, $\Delta$ is a triangle in $\Hb^2$ with one ideal vertex $C$, and the case $m=1$ is degenerate and treated later. Let $a'$ denote the rotation by $\pi$ around the midpoint of $AB$, $b'$ denote the reflection on the side $BC$ and $c'$ denote the reflection on the side $AC$.

Since $b'$ and $c'$ are both reflections and thus fix $C$, in the case of $n\neq \infty$ we obtain for $C$ a cycle transformation $h=c' \circ b' \circ c' \circ b'$. Moreover, the angle between $h(BC)$ and $BC$ is $4\cdot\frac{\pi}{n}  = \frac{4\pi}{n}$. This leads to the relation $(c'\circ b')^n = \text{Id}_{\X}$ for $n \geq 2$. If $n=\infty$, then $b'\circ c'$ has infinite order since $C$ is an ideal vertex.

Now by considering the edge $AB$, we obtain the cycle transformation $h=c' \circ a' \circ b' \circ a'$ with angle between $h(AB)$ and $AB$ equal to $4\cdot \frac{\pi}{2m}= \frac{2\pi}{m}$. This leads to the relation $(c' \circ a' \circ b' \circ a')^m = \text{Id}_{\X}$. We illustrate this in Figure \ref{fig: action type 3} for the case $(m,n)=(2,2)$.

\begin{figure}[h!]
	\centering
	\begin{tikzpicture}[scale=2]
		
		% --- unit length for small square ---
		\def\u{1}
		
		% --- 3x3 small squares ---
		\foreach \i in {0,1,2} {
			\foreach \j in {0,1} {
				
				% --- horizontal edges (red) ---
				\draw[thick, red!70!black] (\i*\u,\j*\u) -- ++(\u,0);
				\draw[thick, red!70!black] (\i*\u,\j*\u+\u) -- ++(\u,0);
				
				% --- vertical edges (red) ---
				\draw[thick, red!70!black] (\i*\u,\j*\u) -- ++(0,\u);
				\draw[thick, red!70!black] (\i*\u+\u,\j*\u) -- ++(0,\u);
				
				% --- diagonals ---
				\draw[thick, blue] (\i*\u,\j*\u) -- ++(\u,\u);
				\draw[thick, green!60!black] (\i*\u+\u,\j*\u) -- ++(-\u,\u);
			}
		}
		
		% --- central top triangle vertices ---
		\coordinate (A) at (1,1);
		\coordinate (B) at (2,1);
		\coordinate (C) at (1.5,1.5);
		
		% --- midpoints ---
		\coordinate (Mab) at ($(A)!0.5!(B)$);
		\coordinate (Mac) at ($(A)!0.5!(C)$);
		\coordinate (Mbc) at ($(B)!0.5!(C)$);
		
		\tikzset{
			rotarrow/.style={
				{Stealth[length=3pt, bend]}-, 
				>=Stealth, 
				thin
			}
		}
		
		\tikzset{
			refarrow/.style={
				<->,               
				{Stealth[length=3pt, bend]}-{Stealth[length=3pt, bend]}, 
				>=Stealth,          
				thin
			}
		}
		
		% --- rotation on horizontal edge AB ---
		\fill (Mab) circle (0.5pt);
		\draw[rotarrow] ($(Mab)+(250:0.12)$) arc (250:-70:0.12); 
		\node[below=0.3] at (Mab) {$a$};
		
		% --- reflections on the other two edges ---
		
		\draw[refarrow] ($(Mac)+(-0.11,0.11)$) -- ($(Mac)+(0.11,-0.11)$); 
		\node[above=0.3] at (Mac) {$b$};

		\draw[refarrow] ($(Mbc)+(-0.11,-0.11)$) -- ($(Mbc)+(0.11,0.11)$); 
		\node[above=0.3] at (Mbc) {$c$};
		
	\end{tikzpicture}
	\caption{The action of $W$ for $(m,n)=(2,2)$ in presentation \ref{type3}.}
	\label{fig: action type 3}
\end{figure}
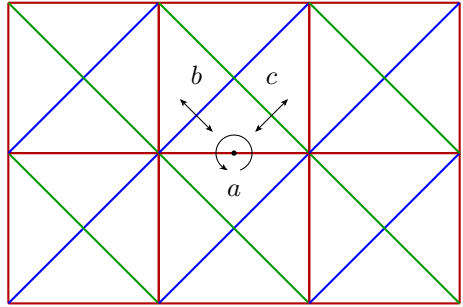

For $(m,n)\neq (1,\infty)$, we thus obtain by Poincaré's Polyhedron Theorem that the involution system $(W,S)$ presented by \ref{type3} is isomorphic to $W'=\langle a',b',c'\rangle$. Moreover, it has Cayley graph dual to the tiling of $\X$ by translates of $\Delta$ and thus isomorphic (as an unlabeled graph) to the Cayley graph of the triangle group $(2m,2m,n)$, the Coxeter companion of $(W,S)$. See, for instance, Figure \ref{fig:Hypo} for the Cayley graph $(W,S)$ when $(m,n)=(3,2)$.

\begin{center}
	\begin{figure}[h!]
		\includegraphics[width=10cm]{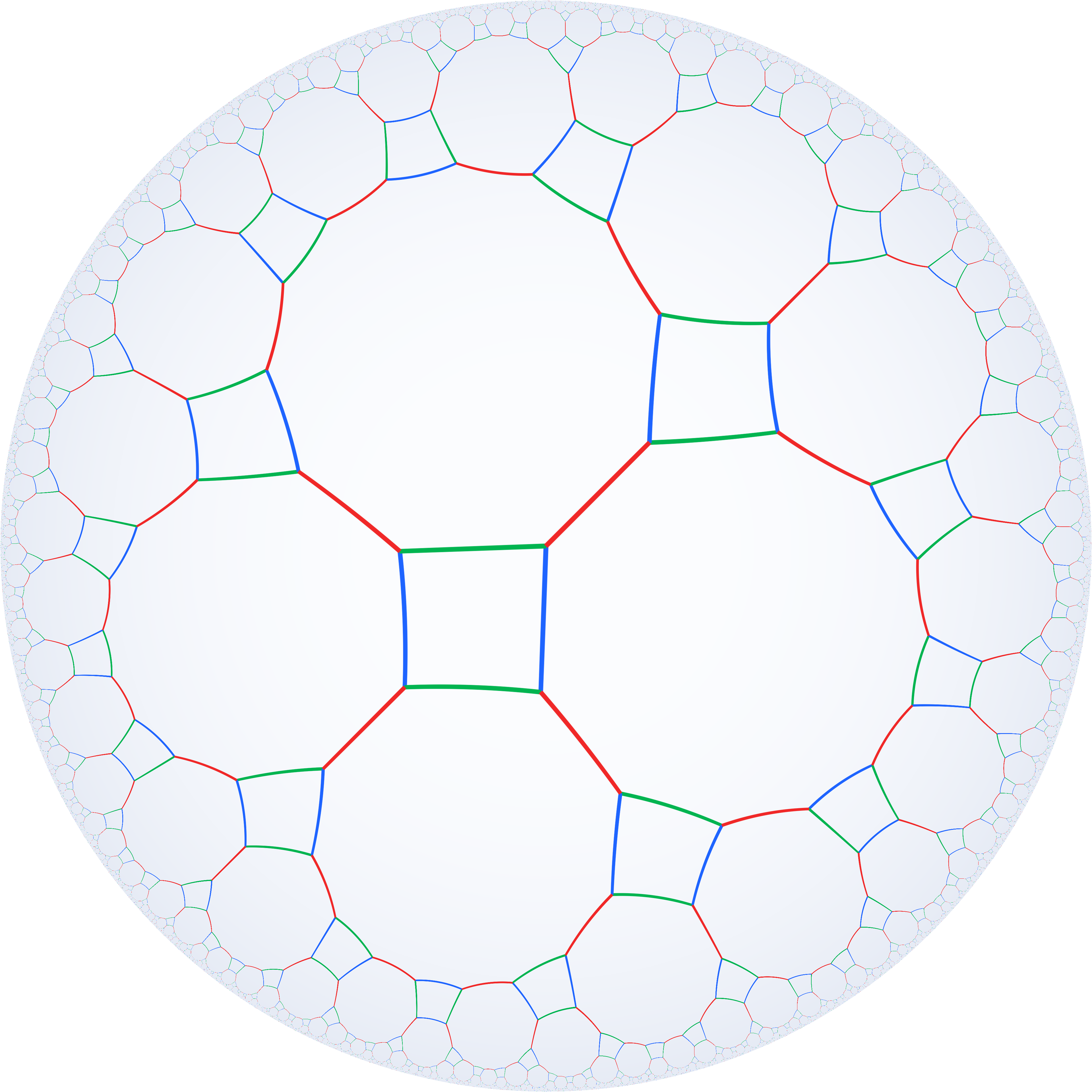}
		\caption{The Cayley graph $\Cay(W,\{a,b,c\})$ for the quasi-Coxeter system: \\$W=\mpair{a, b, c \mid a^2 = b^2 = c^2 = (abac)^3 = (bc)^2 = \id}$.}
		\label{fig:Hypo}
	\end{figure}
\end{center}

The case $(m,n) = (1,\infty)$ can be treated separately and its Cayley graph is depicted in Figure \ref{fig: Presentation 3, n=infty}. In particular, it is isomorphic (as an unlabeled graph) to the Cayley graph of the Coxeter system $\langle a,b,c \, | \, a^2=b^2=c^2 = (ab)^2=(bc)^2=\id\rangle$. Hence, in all cases, $(W,S)$ is a quasi-Coxeter system. This completes the proof of Theorem \ref{rank3-theorem}.
\end{proof}

\begin{figure}
	\centering
	
	% --- FIRST PICTURE (LADDER) ---
	\begin{tikzpicture}[
		edge a/.style={red!70!black, line width=1.5pt},
		edge b/.style={blue, line width=1.5pt},
		edge c/.style={green!60!black, line width=1.5pt},
		vertex/.style={circle,draw=black,fill=blue, thick, inner sep=1.5pt},
		scale=1.5
		]
		
		% Parameters for the ladder - Updated to 5 squares
		\def\cols{5} 
		\def\dist{1} 
		
		% Draw the rungs (a-edges) and rails (b and c edges)
		\foreach \i in {0,...,\cols} {
			\coordinate (top\i) at (\i*\dist, 1);
			\coordinate (bot\i) at (\i*\dist, 0);
			
			% Draw the rungs (a)
			\draw[edge a] (bot\i) -- (top\i);
			
			% Draw the rails (b and c alternating)
			\ifnum\i<\cols
			\pgfmathtruncatemacro{\next}{\i+1}
			\ifodd\i
			\draw[edge c] (top\i) -- (\next*\dist, 1);
			\draw[edge b] (bot\i) -- (\next*\dist, 0);
			\else
			\draw[edge b] (top\i) -- (\next*\dist, 1);
			\draw[edge c] (bot\i) -- (\next*\dist, 0);
			\fi
			\fi
		}
		
		% Draw the vertices
		\foreach \i in {0,...,\cols} {
			\node[vertex] at (top\i) {};
			\node[vertex] at (bot\i) {};
		}
		
		% Labels for the CENTER square (now shifting to the new middle)
		\node[left] at (2, 0.5) {$a$}; 
		\node[right] at (3, 0.5) {$a$};
		\node[above] at (2.5, 1) {$c$};
		\node[below] at (2.5, 0) {$b$};
		
		% Centered Dashed Lines at the ends
		\draw[edge a, dashed, opacity=0.6] (-0.5, 0.5) -- (0, 0.5);
		\draw[edge a, dashed, opacity=0.6] (\cols*\dist, 0.5) -- ++(0.5, 0);
	\end{tikzpicture}

	\caption{The Cayley graph of $\langle a,b,c \, | \, a^2=b^2=c^2 = abac=\id\rangle$.}
	\label{fig: Presentation 3, n=infty}
\end{figure}
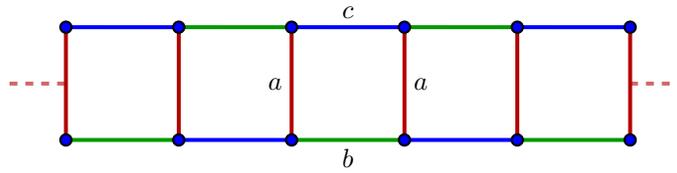

\begin{rem}
	The construction of quasi-Coxeter systems via actions on the space $\X$ by reflections and rotations by $\pi$ around midpoint of sides can be generalized to other polygons and polyhedra. In further work, we plan to present and study this general construction, which we name \emph{polyhedral involution systems}.
\end{rem}

%%%%%%%%%
%%%%%%%
\section{Further Works and Open Problems}
\label{se:FurtherOpen}

We end this article with some results and discussions concerning \EMIS. In Figure~\ref{fig:inclusions}, we survey the different inclusions between classes of involution systems.

\begin{figure}[h!]
\centering
\begin{tikzpicture}[
    node distance=1cm,
    every node/.style={font=\large}
]

\node (IS) {\IS};

\node (MIS) [below left=of IS] {\MIS};
\node (EIS) [below right=of IS] {\EIS};

\node (EMIS) [below=of MIS] {\EMIS};
\node (2-Rec) [below=of EIS] {\tworec};

\node (QCS) [left=of EMIS] {\QCS};
\node (CoxSyst) [left=of QCS] {\CS};

% Inclusion symbols only (no lines)

\node at ($(MIS)!0.5!(IS)$) {\rotatebox{45}{$\subsetneq$}};
\node at ($(EIS)!0.5!(IS)$) {\rotatebox{-45}{$\supsetneq$}};

\node at ($(EMIS)!0.5!(MIS)$) {\rotatebox{90}{$\subsetneq$}};
\node at ($(2-Rec)!0.5!(EIS)$) {\rotatebox{90}{$\subsetneq$}};
\node at ($(2-Rec)!0.5!(EMIS)$) {{$\subsetneq$}};

\node at ($(QCS)!0.5!(EMIS)$) {{$\subsetneq$}};
\node at ($(CoxSyst)!0.5!(QCS)$) {{$\subsetneq$}};

\end{tikzpicture}
\caption{Inclusion relations between classes of involution systems.}
\label{fig:inclusions}
\end{figure}
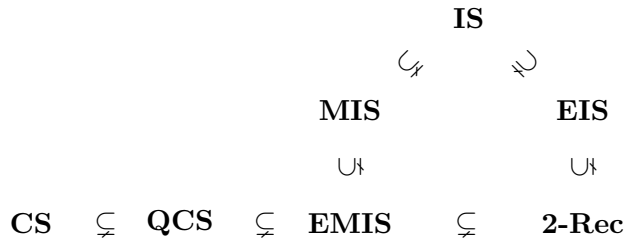

%%%%%
\subsection{A rank 4 \EMIS that is not a quasi-Coxeter system}\label{ss:Rank4}

Consider the involution system $(W,S)$ defined by:
$$
W = \langle a,b,c,d \mid a^2=b^2=c^2=d^2 =(abc)^2=(ad)^2=1\rangle.
$$
Since all relations are of even length, $(W,S)$ is an \EIS by Corollary~\ref{cor:EIS-relations} and the above presentation is a $2$-recognizable presentation of $(W,S)$. 
Let $I=\{a,b,c\}\subset S$ and consider the subgroup $W_I=\langle I\rangle$.  Observe that the involution system $(W_I,I)$ is the \EMIS of Example~\ref{ex:EMIS-A2}. Therefore the only irreducible cycle containing $\id$ needed to obtain the $2$-recognizable involution system above is the cycle of length $4$ with vertices $\id, a,d,ad$ and the companion graph of $(W,S)$ is:
\begin{center}
$\Gamma:$ \hskip 1cm
\begin{tikzpicture}
    [scale=1,
     sommet/.style={inner sep=2pt,circle,draw=black,fill=blue,thick,anchor=base},
     rotate=0,
     baseline = 0]
    \tikzstyle{every node}=[font=\small]
    % Anchor point
    \coordinate (ancre) at (0,-0.0);
    
    % Nodes
    \node  at ($(ancre)+(0,0.6)$) {};
    \node[sommet] (a1) at ($(ancre)+(1,0.5)$) {};
    \node[above=2pt,color=blue] at (a1) {$c$};

    \node[sommet] (a3) at ($(ancre)+(2,0)$) {} edge[thick] node[above,pos=0.35] {$\infty$} (a1);
    \node[below=2pt,color=blue] at (a3) {$d$};

    \node[sommet] (a2) at ($(ancre)+(1,-0.5)$) {} 
    edge[thick] node[auto,swap,right] {} (a1) 
    edge[thick] node[auto,swap,below,pos=0.45] {$\infty$} (a3);
   \node[below=2pt,color=blue] at (a2) {$b$};
    
    \node[sommet] (a4) at (ancre) {} 
    edge[thick] node[auto,swap,right] {} (a1) 
    edge[thick] node[auto,swap,above,pos=0.35] {} (a2);
    \node[below=2pt,color=blue] at (a4) {$a$};
\end{tikzpicture}.
\end{center}

\noindent However $(W,S)$ is not a quasi-Coxeter system. Indeed, $\Cay(W,S)$ is obtained as follows:
\begin{enumerate}
\item Start from $\Cay(W_I,I)$: to each edge labeled by $a$ attached a square for the relation $ad=da$. 
\item Now, at the edge of that square labelled by $a$ that is not in $\Cay(W_I,I)$, attach a copy of $\Cay(W_I,I)$ with adequate labelling of the vertices: we call it a {\em bridge}. 
\item Then do the same process for any copy of $\Cay(W_I,I)$.
\end{enumerate} 
We obtain an infinite graph with an infinite number of disjoint copies $\Cay(W_I,I)$; two copies of $\Cay(W_I,I)$ are linked by a at most one bridge, each corresponding to the relation $ad=da$. This graph is not the Cayley graph of a Coxeter system. Still, since every element of $(W,S)$ determines an unique sequence of bridges, and because $(W_I, I)$ is an \EMIS, it is straightforward to show that $(W,S)$ is also an \EMIS. 

Denote by $(\overline W,S)$ of type $\Gamma$, the Coxeter system corresponding to the Coxeter companion of $(W,S)$. One can show that surprisingly, in this case,  $W(q)=\overline W(q)$. In particular, $W(q)$ is rational. 

%%%%%%%%%%%%%%%%%%%%%
\subsection{Classification problems} In Figure~\ref{fig:inclusions}, we see the different types of involution systems introduced in this article. Of  course, for each of these types, the question of classifying them can be asked. Amongst the first questions that could be asked we propose the the following: find necessary and sufficient conditions for an involution systems to be an \MIS; classify the \EMIS that are quasi-Coxeter systems;  classify the \EMIS that have a given companion graph.

%%%%%%%%%%%%%%%%%%%%%
\subsection{On Growth series and the language of reduced words of \EMIS} The results of \S\ref{ss:Growth} lead us to ask the following questions. 

\begin{question}\label{qu:Growth} Let $(W,S)$ be an \EMIS. Is $W(q)$ rational? Is $W(q)=\overline W(q)$, where $(\overline W,\overline S)$ is the Coxeter-companion of $(W,S)$?
\end{question}

We do not believe that the answer to the second part of the question is positive, but to find a counterexample would be very interesting. The growth series of the \EMIS that is not a quasi-Coxeter system from~\S\ref{ss:Rank4}  provides a positive answer to Question~\ref{qu:Growth} in this case. 

Of course, all classical questions in combinatorial group theory can be asked for an \EMIS in general:  Is $\Red(W,S)$ regular? Is the word problem solvable? Is the conjugacy problem solvable? The answer to those questions is positive in the case of a quasi-Coxeter system, as was shown in \S\ref{ss:Growth} and \S\ref{ss:Automatic}. Moreover, for a general \EMIS $(W,S)$, if every relation in its 2-recognizable presentation has length at least 8, then small cancellation theory tells us that the word problem of $(W,S)$ is solvable \cite[\S5.9]{LyndonScupp}.

\subsection{On Garside shadows of \EMIS and automaticity} \label{ss: Garside}  
The notion of Garside shadows was introduced in~\cite{DyHo16} in the context of finding finite Garside families for Artin-Tits monoids; see also~\cite{DDH14}. This notion can also be defined for any \EMIS:    A {\em Garside shadow of an \EMIS $(W,S)$} is a subset $B$ of $W$ containing $S$ such that:
    \begin{enumerate}[(i)]
    \item if $X\subset B$ is bounded, then $\bigvee_R X\in B$ (join closed);
    \item if $w\in B$ and $v$ is a suffix of $w$, i.e. $v\leq_L w$ in left weak order, then $v\in B$ (suffix closed).
    \end{enumerate}

Since intersections of Garside shadows are Garside shadows, there is a unique {\em smallest Garside shadow} denote by $\tilde S$. Garside shadows plays also an important role in relation to automatic structures of Coxeter systems, see \cite{PrYa24,RegPartition}. They provide in particular {\em voracious languages}, see~\cite{Sa25}, which are the languages used to show that Coxeter systems are biautomatic~\cite{OP22} as discussed in \S\ref{ss:Automatic}. A key result, shown in~\cite{DyHo16}, is that there exist {\em finite} Garside shadows. The proof relies on the fact that the set of {\em elementary walls (or small roots)} is finite~\cite{BH-par}, which is key to exhibit a biautomatic structure for Coxeter systems. As we do not see at this point how to define the concept of walls, or roots, in the context of an \EMIS, the existence of finite Garside shadows might show itself relevant in the study of the problem of automaticity of \EMIS.  

\smallskip

Note that if $(W,S)$ is a quasi-Coxeter system with Coxeter companion $(\overline{W}, \overline{S})$, then the lattice isomorphism $\varphi : \overline{W} \to W$ preserves joins. Hence, if $B$ is a finite Garside shadow in $(\overline{W}, \overline{S})$, then its image $\varphi(B)$ in $(W,S)$ is join closed. However, $\varphi$ does not preserve suffixes, so we cannot immediately conclude that $\varphi(B)$ is a Garside shadow in $(W,S)$. Nonetheless, we have performed some computations in basic examples and it seems to be the case that $\varphi(B)$ is indeed suffix closed, though we were not able to find a proof so far.

%%%%%%%%%%
\subsection{Geometric representation of involution systems} We denote by $O(V,B)$ the group of isometries of a real quadratic space  $(V,B)$.  A {\em geometric representation} of $(W,S)$  is a faithful linear representation $W\to O(V,B)$ such that: (i) $S$ acts by reflections on $V$; (ii) $W$ is discrete in $O(V,B)$. We say then that $W$ is a {\em discrete reflection group over $(V,B)$}. It is well-known that Coxeter systems have geometric representations (even many in the case of indefinite Coxeter systems); see for instance \cite[\S2.3]{DyHo16} and the references therein.

\smallskip
Jon McCammond made us aware of two works in progress (one with Barbara Baumeister and George Neaime and another one with Saeid Azam) on the subject of discrete reflection groups that are not Coxeter groups~ \cite{McC25}. In particular, in their work, the involution system $(W,S)$ in Example~\ref{ex:EMIS-A2} can be realized as a discrete reflection group in a quadratic space $(V,B)$, with basis $(e_1,e_2,e_3)$ of dimension $3$ where
$$
B=
\begin{pmatrix}
1 & 1 & 1 \\
1 & 1 & 1 \\
1 & 1 & 1
\end{pmatrix}.
$$
Then, $S$ is the set of reflections $s_{e_i}(e_j)=e_j-2e_i$ and $W=\langle S\rangle$. Therefore  $W$ is a discrete reflection group with this geometric representation. This group also appear in conjunction with the notion of {\em Riemannian symmetric spaces} (by replacing $s_{e_i}$ by $-s_{e_i}$): it is then a discrete lattice in the continuous Lie Group of $\mathbb R^3$ generated by {\em point reflections} (rotations of angle $\pi$).
\smallskip

In the case of an \EMIS $(W,S)$ of rank $3$, we constructed in the proof of Theorem~\ref{rank3-theorem} a representation of $W$ as a discrete subgroup of  $\Isom(\X)$ generated by reflections and rotations of angle $\pi$; here $\Isom(\X)$ denotes the group of isometries on $\X = \Sb^2, \E^2$ or $\Hb^2$. It is well-known that there is a real quadratic space $(V,B)$ of dimension $3$ such that:
\begin{itemize}
\item $\Isom(\Sb^2)\cong O(3)=O(V,B)$, the orthogonal group over the Euclidean vector space $V=\mathbb R^3$ and $B$ is a scalar product;
\item $\Isom(\E^2)\cong O(2,0)=O(V,B)$, where $(V,B)$ is a positive degenerate quadratic space of dimension $3$  with signature $(2,0)$; 
\item $\Isom(\Hb^2)\cong O^+(2,1)\leq O(2,1)=O(V,B)$, where $(V,B)$ is the Lorentzian space. 
\end{itemize}

Now any reflection (resp. rotation)  on $\X$ correspond to a reflection (resp. rotation)  in $O(V,B)$. Moreover, if $r$ is a rotation of angle $\pi$, then $-r$ is a reflection in $O(V,B)$. Therefore $W$ is a discrete reflection group in $O(V,B)$. We obtain  the following corollary.

\begin{cor}[of the proof of Theorem~\ref{rank3-theorem}] Any \EMIS of rank at most $3$ admits a geometric representation.
\end{cor}

The geometric representation provided above for Example~\ref{ex:EMIS-A2} is not the one we obtain in the proof of Theorem~\ref{rank3-theorem} since the signature of the considered quadratic space of dimension $3$ is $(1,0)$. 

\begin{question}\label{qu:6} For which involution system are there geometric representations as discrete reflection groups in a quadratic space? Can we classify them?
\end{question}

%%%%%
\subsection{On standard subgroups of \EMIS}  Let $(W,S)$ be an \EMIS. Let $I\subseteq S$, and $W_I=\mpair{I}$. We do not know if $(W_I,I)$ is necessarily an \EMIS.  

\begin{question} \label{qu: parabolic subgroup} Is $W_I$ a sub-semilattice of $(W,\leq_R)$? More precisely: 
	\begin{enumerate}
		\item If $X\subseteq W_I$ nonempty, is $\bigwedge_R X\in W_I$?
		\item If $X\subseteq W_I$ is bounded, is $\bigvee_R X\in W_I$?
	\end{enumerate}
\end{question}

If the answer is positive, we would have an analog for \EMIS of the concept of {\em standard subgroups} as in the case of Coxeter systems. If the answer is negative, $(W_I,I)$ might still be an \EMIS but with a lattice structure that is not inherited from the one of $(W,S)$.

Observe that, contrary to Coxeter systems, if $w=s_1\cdots s_k$ is a reduced word for $w \in W_I$ with $s_i\in S$, it is not true in general that $s_i\in I$. Consider the finite \EMIS given by $W=\langle a,b,c \mid a^2,b^2,c^2,abac,(bc)^2 \rangle$. Then $c=aba\in W_{a,b}$ but the reduced word $c\notin \{a,b\}$. 

\smallskip
Another natural question to ask in regards to the lattice structure of an EMIS is about the existence of an equivalent of a longest element for finite $\EMIS$: Let $(W,S)$ be an \EMIS.
	\begin{enumerate}
		\item Under which assumptions is $(W,\leq_R)$ bounded? (And therefore a complete lattice.)
		\item If $S$ is bounded in $(W,\leq_R)$, is it true that $W$ is finite? Is it true  that $\bigvee_R S$ is a greatest bound in $(W,\leq_R)$?
	\end{enumerate}    
Recall that the the alternate group $\mathcal A_5$ in Example~\ref{ss:Alt5}  is a finite \MIS, but is neither an \EMIS nor a bounded poset.

%%%%%%%%%%%%%%%
\subsection{Mediangle graphs and \EMIS}\label{ss:Mediangleness}
    Tatiana Smirnova-Nagnibeda and Megan Howarth have pointed out to us possible connections between mediangle graphs and Cayley graphs of \EMIS. Mediangle graphs were introduced in~\cite{genevois2022rotation} by Genevois, for which he described notions of hyperplanes and separation/inversion sets (\cite[Theorem~3.9]{genevois2022rotation}). Another interesting property of groups with a mediangle Cayley graph is a version of the Matsumoto theorem: all geodesics from $\id$ to $w$ are equivalent up to {\em flips} (\cite[Lemma~3.3]{genevois2022rotation}), which are braid moves in Coxeter systems.
    
    In~\cite[Proposition~3.2]{genevois2022rotation}, Genevois showed that any median graph is mediangle. So any \EMIS with median Cayley graph is mediangled; see~\S\ref{Median subsection} for the discussion on median graphs and \EMIS.
    Moreover, all examples of \EMIS presented in this article, and all the ones we have investigated so far, have mediangle Cayley graph. This leads to the following question, which is the subject of ongoing research with Tatiana Smirnova-Nagnibeda and Megan Howarth.

    \begin{question}\label{qu:Mediangleness}
        If $(W,S)$ is an \EMIS, is $\Cay(W,S)$ a mediangle graph?
    \end{question}
    
    In \cite[Definition~1.4]{genevois2022rotation}, the only properties that one has to show to answer Question \ref{qu:Mediangleness} in the affirmative are the {\em cycle condition} and the {\em intersection of even cycles} condition. 
    
    Let $(W,S)$ be an \EMIS. The {\em intersection of even cycles condition} can be restated as the condition that the set of convex cycles weakly intersects. With Tatiana Smirnova-Nagnibeda and Megan Howarth, we already have a proof of this fact.
    
    The {\em cycle condition} can be restated in terms of the weak order: 

    \begin{quote}
        \noindent {\bf The cycle condition for an \EMIS.} $(W,S)$ satisfies the cycle condition if for any $w\in W$ and $s\neq t\in S$ such that $ws,wt\leq_R w$ the cycle containing $w,ws,wt$ and $ws\wedge_R wt$ is convex.   
    \end{quote}

The isometry $w^{-1}$ on $\Cay(W,S)$ maps a cycle containing $w,ws,wt$ and $ws\wedge_R wt$ to a cycle containing $\id, s, t$ and $s\vee_R t$. So to show that $(W,S)$ satisfies the cycle condition one can show that Question~\ref{qu:Join} has a positive answer.

%%%%%%
\subsection{On CAT(0) cellular structures for \EMIS}

Every Coxeter group $\overline{W}$ acts properly and cocompactly by isometries on a contractible CAT(0) piecewise Euclidean cell complex $\Sigma$ called the \emph{Davis complex} of $\overline{W}$ \cite{Davisbook, Moussong}. This has many important consequences for the geometric and topological study of Coxeter groups.

This leads to the following question.

\begin{question} \label{qu: Davis complex}
	Is it possible to construct an analog of the Davis complex for an \EMIS? Are \EMIS \emph{CAT(0)}?
\end{question}

In \cite{chepoicellular}, the authors show that any bipartite mediangle graph can be endowed with the structure of a contractible cell complex. Hence, if it is the case that Cayley graphs of \EMIS are mediangle, one would obtain such a contractible complex for an \EMIS. However, it is unclear if one can endow these complexes with a CAT(0) metric.

\subsection{On convex subgraphs of Coxeter systems}

By \cite{HaglundWise}, any median graph can be realized as a convex subgraph of the Cayley graph of some right-angled Coxeter system. One natural generalization of this result that Genevois asks in \cite{genevois2022rotation} is: Which bipartite mediangle graphs can be realised as convex subgraphs of the Cayley graph of some Coxeter system?

Even though we still don't know if the Cayley graph of every \EMIS is mediangle, all examples of \EMIS in this paper satisfy this property. That is, for any \EMIS $(W,S)$ appearing in this work we can find some Coxeter system $(\overline{W}, \overline{S})$ such that $\Cay(W,S)$ is a convex subgraph of $\Cay(\overline{W}, \overline{S})$. This lead us to ask the following question.

 \begin{question} \label{qu: convex subgraph}
 	Can the Cayley graph of any \EMIS be realized as a convex subgraph of the Cayley graph of some Coxeter system?
 \end{question}
 
 If the answer to Question~\ref{qu: convex subgraph} is positive, this would have notable consequences. For instance, it would yield a positive answer to Question \ref{qu: parabolic subgroup} and give the notion of parabolic subgroups of \EMIS. Another consequence would be the finiteness of cone types of an \EMIS, thus giving a positive answer to the first part of Question~\ref{qu:Growth} and the question of the regularity of the language $\Red(W,S)$.

% ---------- Column types ----------
\newcolumntype{C}{>{\centering\arraybackslash}m{3cm}} % Diagram columns
\newcolumntype{P}{>{\raggedright\arraybackslash}m{7.5cm}} % Presentation column

\setlength{\tabcolsep}{10pt} % Horizontal padding
\renewcommand{\arraystretch}{2} % Row height

\clearpage

\begin{table}[ht]
	\centering
	\small
	\begin{adjustbox}{width=160mm, center}
		\begin{tabular}{l C C P}
			\toprule
			& Companion graph & Coxeter type & Presentation \\
			\midrule
			
			\multirow[c]{10}{*}{Finite}
			& \nodesA & \nodesA & $\langle a \mid a^2 \rangle$ \\
			& \nodesAB & \nodesAB & $\langle a,b \mid a^2,b^2,(ab)^2 \rangle$ \\
			& \edgeAB{m} & \edgeAB{m} & $\langle a,b \mid a^2,b^2,(ab)^m \rangle,\  m\ge1$ \\
			& \multirow{2}{*}{\nodesABC} & \nodesABC & $\langle a,b,c \mid a^2,b^2,c^2,(ab)^2,(ac)^2,(bc)^2 \rangle$ \\
			& & \triangleABCC{4}{4} & $\boxed{\langle a,b,c \mid a^2,b^2,c^2,abac,(bc)^2 \rangle}$ \\
			& \multirow{2}{*}{\splitABC{n}} & \splitABC{n} & $\langle a,b,c \mid a^2,b^2,c^2,(ab)^2,(ac)^2,(bc)^n \rangle, \ n\ge3$ \\
			& & \triangleABC{2n}{2n}{n} & $\boxed{\langle a,b,c \mid a^2,b^2,c^2,abac,(bc)^n \rangle,\ n\ge3}$ \\
			& \chainABC{}{} & \chainABC{}{} & $\langle a,b,c \mid a^2,b^2,c^2,(ab)^3,(bc)^3,(ac)^2 \rangle$ \\
			& \chainABC{4}{} & \chainABC{4}{} & $\langle a,b,c \mid a^2,b^2,c^2,(ab)^4,(bc)^3,(ac)^2 \rangle$ \\
			& \chainABC{5}{} & \chainABC{5}{} & $\langle a,b,c \mid a^2,b^2,c^2,(ab)^5,(bc)^3,(ac)^2 \rangle$ \\
			\midrule
			
			\multirow[c]{10}{*}{Euclidean}						
			& \edgeAB{\infty} & \edgeAB{\infty} & $\langle a,b \mid a^2,b^2 \rangle$ \\
			\addlinespace[8pt]
			& \multirow{2}{*}{\splitABC{\infty}} & \splitABC{\infty} & $\langle a,b,c \mid a^2,b^2,c^2,(ab)^2,(ac)^2 \rangle$ \\
			& & \triangleABC{\infty}{\infty}{\infty} & $\boxed{\langle a,b,c \mid a^2,b^2,c^2,abac \rangle}$ \\
			\addlinespace[8pt]
			& \multirow{3}{*}{\triangleABC{}{}{} } & \triangleABC{}{}{} & $\langle a,b,c \mid a^2,b^2,c^2,(ab)^3,(ac)^3,(bc)^3 \rangle$ \\
			& &  \multirow{2}{*}{\triangleABC{\infty}{\infty}{\infty}} & $\boxed{\langle a,b,c \mid a^2,b^2,c^2,(abc)^2 \rangle}$ \\
			& & & $\boxed{\langle a,b,c \mid a^2,b^2,c^2,abacbc \rangle}$ \\
			\addlinespace[8pt]
			& \multirow{3}{*}{\triangleABCC{4}{4} } & \triangleABCC{4}{4} & $\langle a,b,c \mid a^2,b^2,c^2,(ab)^4,(ac)^4,(bc)^2 \rangle$ \\
			& &  \triangleABCC{\infty}{\infty} & $\boxed{\langle a,b,c \mid a^2,b^2,c^2,(abac)^2,(bc)^2 \rangle}$ \\
			\addlinespace[8pt]
			& \chainABC{6}{} & \chainABC{6}{} & $\langle a,b,c \mid a^2,b^2,c^2,(ab)^6,(bc)^3,(ac)^2 \rangle$ \\

			\bottomrule
		\end{tabular}
	\end{adjustbox}
	\caption{Classification of finite/spherical and Euclidean \EMIS of rank at most 3. The \EMIS\ {\em that are not} Coxeter systems are boxed.}
	\label{table: Classification in rank 3}
\end{table}

\begin{table}[ht]
	\centering
	\small
	\begin{adjustbox}{width=180mm, center}
		\begin{tabular}{l C C P}
			\toprule
			Family & Companion graph & Coxeter type & Presentation \\
			\midrule

			\multirow[c]{13}{*}{Hyperbolic}
			& \multirow{3}{*}{\triangleABC{3m}{3m}{3m}} & \triangleABC{3m}{3m}{3m} & $\langle a,b,c \mid a^2,b^2,c^2,(ab)^{3m},(bc)^{3m},(ac)^{3m} \rangle$, $m\ge1$ \\
			& & \multirow{2}{*}{\triangleABC{\infty}{\infty}{\infty}} & $\boxed{\langle a,b,c \mid a^2,b^2,c^2,(abc)^{2m} \rangle,\ m\ge2}$ \\
			& & & $\boxed{\langle a,b,c \mid a^2,b^2,c^2,(abacbc)^m \rangle,\ m\ge2}$ \\
			\addlinespace[8pt]
			
			& \multirow{3}{*}{\chainABC{2m}{2m}}&\chainABC{2m}{2m} & $\langle a,b,c \mid a^2,b^2,c^2,(ab)^{2m},(bc)^{2m}, (ac)^2 \rangle$, $m\ge3$ \\
			
			&& \triangleABCC{\infty}{\infty} & $\boxed{\langle a,b,c \mid a^2,b^2,c^2,(abac)^m,(bc)^2 \rangle,\ m\ge3}$ \\
		\addlinespace[8pt]
			& \multirow{3}{*}{\triangleABC{2m}{2m}{n}}  & \triangleABC{2m}{2m}{n} & $\langle a,b,c \mid a^2,b^2,c^2,(ab)^{2m},(ac)^{2m},(bc)^n \rangle$, $m\ge2$, $n\ge3$ or $\infty$ \\
			& & \triangleABC{\infty}{\infty}{n} & $\boxed{\langle a,b,c \mid a^2,b^2,c^2,(abac)^m,(bc)^n \rangle,\ m\ge2,\ n\ge3\ \text{or} \ \infty}$ \\
			\addlinespace[8pt]
						
			&\chainABC{m}{} & \chainABC{m}{} &$\langle a,b,c \mid a^2,b^2,c^2,(ab)^m,(bc)^3,(ac)^2 \rangle,\ m\ge7\ \text{or} \ \infty$\\
			
			&\chainABC{m}{n} & \chainABC{m}{n} &$\langle a,b,c \mid a^2,b^2,c^2,(ab)^m,(bc)^n,(ac)^2 \rangle,\ m\ge5,\ n\ge4\ \text{or} \ \infty$\\
			
			&\triangleABC{m}{n}{p} & \triangleABC{m}{n}{p} &$\langle a,b,c \mid a^2,b^2,c^2,(ab)^m,(bc)^n,(ac)^p \rangle,\ m\ge4,\ n,p\ge3\ \text{or} \ \infty$\\
			
			\bottomrule
		\end{tabular}
	\end{adjustbox}
	\caption{Classification of hyperbolic \EMIS of rank at most 3. The \EMIS\ {\em that are not} Coxeter systems are boxed.}
	\label{table2: Classification in rank 3}
\end{table}

\clearpage

\renewcommand{\arraystretch}{1} % reset to default

\bibliography{refs.bib}{} 
\bibliographystyle{plain} 

\end{document}